\documentclass[a4paper]{amsart}

\usepackage[T1]{fontenc}
\usepackage[utf8]{inputenc}

\usepackage[english]{babel}

\usepackage{amsmath,mathtools,amssymb,extarrows,mathrsfs,amsthm}
\usepackage{tikz-cd}
\tikzcdset{scale cd/.style={every label/.append style={scale=#1},
		cells={nodes={scale=#1}}}}
\usepackage{xcolor}
\usepackage{float}
\usepackage{hyperref}
\usepackage{bm}

\calclayout

\theoremstyle{definition}
\newtheorem{defi}{Definition}[section]
\newtheorem{rem}[defi]{Remark}
\newtheorem{ex}[defi]{Example}

\theoremstyle{plain}
\newtheorem{prop}[defi]{Proposition}
\newtheorem{lemma}[defi]{Lemma}
\newtheorem{cor}[defi]{Corollary}
\newtheorem{thm}[defi]{Theorem}

\newcommand{\A}{\mathbb{A}}
\newcommand{\D}{\mathbb{D}}
\newcommand{\C}{\mathbb{C}}
\newcommand{\E}{\mathbb{E}}
\newcommand{\PP}{\mathbb{P}}

\newcommand{\R}{\mathbb{R}}
\newcommand{\Z}{\mathbb{Z}}
\newcommand{\DD}{\mathrm{D}}
\newcommand{\EE}{\mathrm{E}}
\newcommand{\RR}{\mathrm{R}}
\renewcommand{\k}{\Bbbk}
\newcommand{\cA}{\mathcal{A}}

\newcommand{\cD}{\mathcal{D}}
\newcommand{\cE}{\mathcal{E}}
\newcommand{\cF}{\mathcal{F}}
\newcommand{\cG}{\mathcal{G}}
\newcommand{\cL}{\mathcal{L}}
\newcommand{\cM}{\mathcal{M}}
\newcommand{\cN}{\mathcal{N}}
\newcommand{\cO}{\mathcal{O}}
\newcommand{\cR}{\mathcal{R}}

\newcommand{\cX}{\mathcal{X}}
\newcommand{\cY}{\mathcal{Y}}

\newcommand{\sP}{\mathsf{P}}

\newcommand{\Mod}[1]{\mathrm{Mod}(#1)}
\newcommand{\ModRc}[1]{\mathrm{Mod}_{\R\text{-}\mathrm{c}}(#1)}

\newcommand{\ModCc}[1]{\mathrm{Mod}_{\C\text{-}\mathrm{c}}(#1)}
\newcommand{\ModcohD}[1]{\mathrm{Mod}_{\mathrm{coh}}(\mathcal{D}_{#1})}
\newcommand{\ModholD}[1]{\mathrm{Mod}_{\mathrm{hol}}(\mathcal{D}_{#1})}
\newcommand{\ModrhD}[1]{\mathrm{Mod}_{\mathrm{rh}}(\mathcal{D}_{#1})}
\newcommand{\Isuban}[1]{\mathrm{I}_{\mathrm{suban}}(#1)}
\newcommand{\indlim}[1]{\mathop{``\!\varinjlim\! "}\limits_{#1}}
\newcommand{\DbD}[1]{\mathrm{D}^\mathrm{b}(\mathcal{D}_{#1})}
\newcommand{\DbrhD}[1]{\mathrm{D}^\mathrm{b}_\mathrm{rh}(\mathcal{D}_{#1})}
\newcommand{\DbholD}[1]{\mathrm{D}^\mathrm{b}_\mathrm{hol}(\mathcal{D}_{#1})}
\newcommand{\DbcohD}[1]{\mathrm{D}^\mathrm{b}_\mathrm{coh}(\mathcal{D}_{#1})}
\newcommand{\Dbk}[1]{\mathrm{D}^\mathrm{b}(\k_{#1})}

\newcommand{\DbL}[1]{\mathrm{D}^\mathrm{b}(L_{#1})}

\newcommand{\DbRck}[1]{\mathrm{D}^\mathrm{b}_{\R\text{-}\mathrm{c}}(\k_{#1})}
\newcommand{\DbRcK}[1]{\mathrm{D}^\mathrm{b}_{\R\text{-}\mathrm{c}}(K_{#1})}

\newcommand{\DbCcC}[1]{\mathrm{D}^\mathrm{b}_{\C\text{-}\mathrm{c}}(\C_{#1})}
\newcommand{\DbCck}[1]{\mathrm{D}^\mathrm{b}_{\C\text{-}\mathrm{c}}(\k_{#1})}
\newcommand{\DbIC}[1]{\mathrm{D}^\mathrm{b}(\mathrm{I}\C_{#1})}
\newcommand{\DbIk}[1]{\mathrm{D}^\mathrm{b}(\mathrm{I}\k_{#1})}

\newcommand{\EbIk}[1]{\mathrm{E}^\mathrm{b}(\mathrm{I}\k_{#1})}
\newcommand{\EzIk}[1]{\mathrm{E}^0(\mathrm{I}\k_{#1})}
\newcommand{\EbIC}[1]{\mathrm{E}^\mathrm{b}(\mathrm{I}\C_{#1})}

\newcommand{\EbIK}[1]{\mathrm{E}^\mathrm{b}(\mathrm{I}K_{#1})}
\newcommand{\EbIL}[1]{\mathrm{E}^\mathrm{b}(\mathrm{I}L_{#1})}
\newcommand{\EbRcIC}[1]{\mathrm{E}^\mathrm{b}_{\mathbb{R}\textnormal{-}\mathrm{c}}(\mathrm{I}\C_{#1})}

\newcommand{\EbRcIL}[1]{\mathrm{E}^\mathrm{b}_{\mathbb{R}\textnormal{-}\mathrm{c}}(\mathrm{I}L_{#1})}
\newcommand{\EbRcIK}[1]{\mathrm{E}^\mathrm{b}_{\mathbb{R}\textnormal{-}\mathrm{c}}(\mathrm{I}K_{#1})}
\newcommand{\EbRcIk}[1]{\mathrm{E}^\mathrm{b}_{\mathbb{R}\textnormal{-}\mathrm{c}}(\mathrm{I}\k_{#1})}
\newcommand{\EzRcIk}[1]{\mathrm{E}^0_{\mathbb{R}\textnormal{-}\mathrm{c}}(\mathrm{I}\k_{#1})}
\newcommand{\EzRcIL}[1]{\mathrm{E}^0_{\mathbb{R}\textnormal{-}\mathrm{c}}(\mathrm{I}L_{#1})}
\newcommand{\EzRcIK}[1]{\mathrm{E}^0_{\mathbb{R}\textnormal{-}\mathrm{c}}(\mathrm{I}K_{#1})}

\newcommand{\Amod}{\cA^\mathrm{mod}_{\widetilde{X}}}

\newcommand{\DR}[1]{\mathcal{DR}_{#1}}
\newcommand{\Sol}[1]{\mathcal{S}ol_{#1}}

\newcommand{\DRE}[1]{\mathcal{DR}^\EE_{#1}}

\newcommand{\SolE}[1]{\mathcal{S}ol^\EE_{#1}}

\newcommand{\RIhom}{\mathrm{R}\mathcal{I}hom}
\newcommand{\RHomE}{\RR\mathcal{H}om^\EE}
\newcommand{\DE}[1]{\mathrm{D}^\EE_{#1}}
\newcommand{\RHom}{\RR\mathcal{H}om}
\newcommand{\Hom}{\mathcal{H}om}
\newcommand{\sHom}{\mathrm{Hom}}

\newcommand{\Ihom}{\mathcal{I}hom}

\newcommand{\Db}[1]{\mathcal{D}b_{{#1}}}

\newcommand{\DbE}{\mathcal{D}b^{\mathrm{E}}_{X}}
\newcommand{\Dbt}[1]{\mathcal{D}b^\mathrm{t}_{#1}}

\newcommand{\Cinft}[1]{\mathscr{C}^{\infty,\mathrm{t}}_{#1}}
\renewcommand{\OE}[1]{\cO^\EE_{#1}}

\newcommand{\Prop}[1]{\mathrm{P}_X(#1)}

\newcommand{\XD}{{\wt{X}}}

\newcommand{\Xc}{\overline{X}}
\newcommand{\Yc}{\overline{Y}}

\newcommand{\Xsac}{X_{sa}^c}
\newcommand{\sa}{sa}

\newcommand{\beps}{\bm{\varepsilon}}

\newcommand{\defeq}{\vcentcolon=}
\newcommand{\wt}[1]{\widetilde{#1}}
\newcommand{\iso}{\simeq}
\newcommand{\conv}{\overset{+}{\otimes}}

\newcommand{\e}[2]{e^{#1}_{#2}}

\newcommand{\To}{\longrightarrow}
\newcommand{\ToPO}{\overset{+1}{\longrightarrow}}

\newcommand{\real}{\mathop{\mathrm{Re}}}

\newcommand{\id}{\mathrm{id}}
\newcommand{\supp}{\mathop{\mathrm{supp}}}
\newcommand{\singsupp}{\mathop{\mathrm{sing\,supp}}}

\newcommand{\Ltens}[1]{\overset{\mathrm{L}}{\otimes}_{#1}}
\newcommand{\LtensD}[1]{\overset{\mathrm{L}}{\otimes}_{\mathcal{D}_{#1}}}
\newcommand{\tensOD}{\overset{\mathrm{D}}{\otimes}}

\newcommand{\LtensO}[1]{\overset{\mathrm{L}}{\otimes}_{\mathcal{O}_{#1}}}

\newcommand{\piLK}[1]{\pi^{-1}L_{#1}\otimes_{\pi^{-1}K_{#1}}}

\title[Kashiwara conjugation and enhanced Riemann--Hilbert]{Kashiwara conjugation and the enhanced Riemann--Hilbert correspondence}
\author{Andreas Hohl}
\thanks{The author's research was funded by the Deutsche Forschungsgemeinschaft (DFG, German Research Foundation), Projektnummer 465657531. The author was also supported by the grant G0B3123N from the Fonds voor Wetenschappelijk Onderzoek -- Vlaanderen (FWO, Research Foundation -- Flanders).}
\address{Université Paris Cité and Sorbonne Université, CNRS, IMJ-PRG, F-75013 Paris, France\newline \text{}\quad
\textnormal{\textit{Current address:}} KU Leuven, Departement Wiskunde, Celestijnenlaan 200B, B-3001 Leuven, Belgium.}
\email{andreas.hohl@kuleuven.be}

\begin{document}
	
\begin{abstract}
	We study some aspects of conjugation and descent in the context of the irregular Riemann--Hilbert correspondence of D'Agnolo--Kashiwara. First, we give a proof of the fact that Kashiwara's conjugation functor for holonomic D-modules is compatible with the enhanced de Rham functor. Afterwards, we work out some complements on Galois descent for enhanced ind-sheaves, slightly generalizing results obtained in previous joint work with Barco, Hien and Sevenheck. Finally, we show how local decompositions of an enhanced ind-sheaf into exponentials descend to lattices over smaller fields. This shows in particular that a structure of the enhanced solutions of a meromorphic connection over a subfield of the complex numbers has implications on its generalized monodromy data (notably the Stokes matrices), generalizing and simplifying an argument given in our previous work.
\end{abstract}

	\maketitle

\section{Introduction}

In general, a Riemann--Hilbert correspondence is an equivalence of categories between some category of differential systems (such as integrable connections, meromorphic connections, or D-modules) and some category of topological objects (such as local systems, perverse sheaves, Stokes-filtered local systems, or enhanced ind-sheaves). In the context of differential equations in one or several complex variables (i.e.\ in the theory of connections or D-modules on a complex algebraic variety or complex manifold), these topological objects are a priori defined over the field of complex numbers (as their field of coefficients).

This being said, three questions immediately come to our mind:
\begin{itemize}
	\item[(1)] How does the Riemann--Hilbert functor behave with respect to the natural complex conjugation on the target?
	\item[(2)] How can we detect if the target object is already defined over a subfield of $\C$?
	\item[(3)] What implications does such a structure over a smaller field have?
\end{itemize}

The first question was studied by M.\ Kashiwara in \cite{KasConj} in the case of regular holonomic D-modules. At the time, it was known that the de Rham functor induces an equivalence between the derived category of regular holonomic D-modules and the derived category of constructible sheaves (see \cite{KasRHreg}), and in \cite{KasConj} a conjugation functor on the category of D-modules was defined that corresponds to complex conjugation on constructible sheaves via this equivalence. Around 30 years later, in \cite{DK16}, A.\ D'Agnolo and M.\ Kashiwara generalized Kashiwara's equivalence from \cite{KasRHreg} to (not necessarily regular) holonomic D-modules. It is the aim of the first part (Section~\ref{sec:KashiwaraConj}) of this article to show that the conjugation functor of \cite{KasConj} is -- not surprisingly -- still compatible with this new equivalence of categories (Theorem~\ref{prop:DREconj}). Indeed, many parts of the proof are due to important previous studies of holonomic D-modules and the Hermitian duality functor, notably by C.\ Sabbah, K.\ Kedlaya and T.\ Mochizuki.

In the second part (Section~\ref{sec:GaloisDescentEb}) of the paper, we give some complements on a study of the second question that has been initiated in \cite{BHHS22}. In loc.~cit., together with D.\ Barco, C.\ Sevenheck and M.\ Hien, we studied the question of when topological data associated to hypergeometric differential systems are defined over a subfield of $\C$. For this, we used a technique called \emph{Galois descent}, whose underlying idea is the following: Given a finite Galois extension $L/K$ and an object over $L$ (e.g.\ a sheaf of $L$-vector spaces), then in order to find a structure of this object over $K$ (i.e.\ an object over $K$ that yields our given object after extending scalars), it should be enough to find suitable isomorphisms with all its Galois conjugates. In our joint work, we developed some statements in this direction for sheaves and enhanced ind-sheaves. In particular, we showed that Galois descent is possible for $\R$-constructible enhanced ind-sheaves concentrated in one degree on compact spaces. We will reformulate and slightly generalize this statement here, dropping the compactness assumption (Theorem~\ref{thm:GaloisDescentEb}).

Finally, in the last part (Section~\ref{sec:Monodromy}), we give a consequence of the existence of $K$-structures (for an arbitrary subfield $K\subset \C$) in the irregular setting, thus addressing Question (3): We show that a $K$-structure on an enhanced ind-sheaf associated to a meromorphic connection on a Riemann surface via the irregular Riemann--Hilbert correspondence implies that its Stokes matrices (and indeed all its generalized monodromy data) can be defined over $K$. In the case of hypergeometric systems, this was already done in \cite[§5]{BHHS22}. Our proof here is, however, valid in general and simplifies significantly the quite involved and slightly artificial argument given in our previous work. To this purpose, we study the topological counterpart of meromorphic connections on complex curves (already introduced in \cite{DKmicrolocal}), which we call \emph{enhanced ind-sheaves of meromorphic normal form}, and prove that this property descends to a $K$-lattice (Theorem~\ref{thm:HLTdescent}). Our argument makes use of the theory of subanalytic sheaves.

\subsection*{Acknowledgements} I would like to thank Andrea D'Agnolo and Claude Sabbah for several helpful discussions and hints, in particular concerning the part on Hermitian duality. I also thank Pierre Schapira for some useful conversations. I am grateful to Philip Boalch for his explanations that brought new insights into the Stokes phenomenon to me. I want to thank Davide Barco, Marco Hien and Christian Sevenheck who, in many discussions related to our common work \cite{BHHS22}, inspired me to develop further results in this direction. I am also grateful to Takuro Mochizuki for answering my questions during the preparation of \cite{BHHS22}, the answers to which have inspired this work in various places. Finally, I thank the Simons Center for Geometry and Physics, Stony Brook, for hosting me in the course of the program ``The Stokes phenomenon and its Applications in Mathematics and Physics'' in May and June 2023, during which parts of the research for this work have been performed.
	
\section{Background and notation}
In this work, we want to study objects related to Riemann--Hilbert correspondences for holonomic D-modules, and we will mainly use the approach via (enhanced) ind-sheaves here, great parts of which have been developed in \cite{KS01} and \cite{DK16}. We will therefore need some analytic and topological notions, which we briefly recall in the following, referring to the existing literature for further details.

\subsection{Topological spaces and manifolds} Even if not strictly necessary in all places, we will assume all our topological spaces to be \emph{good}, i.e.\ Hausdorff, locally compact, second countable and of finite flabby dimension. This is especially important when it comes to the construction of proper direct images and exceptional inverse images, and since we are mainly interested in all these objects in the context of the Riemann--Hilbert correspondence on complex manifolds, goodness is not a restriction.

For a complex manifold $X$ with structure sheaf $\cO_X$, we can consider its \emph{complex conjugate} manifold $\Xc$, having the same underlying topological space, but the structure sheaf $\cO_{\Xc}$ is the sheaf of antiholomorphic functions (with respect to the original complex structure). Also, $X$ has an underlying real analytic manifold, which is the same as that of $\overline{X}$. It is often denoted by $X_\R$ to emphasize the real structure, but we will mostly just denote it by $X$ if it is clear from the context. Note also that a morphism $f\colon X\to Y$ of complex manifolds induces a morphism of complex manifolds $\overline{f}\colon \overline{X}\to\overline{Y}$ (which is the same as a morphism of the underlying topological spaces).

\subsection{Holonomic D-modules}\label{sec:Holonomic}
We briefly recall some notions and notation from the theory of D-modules, and we refer to standard references such as \cite{Bjo}, \cite{HTT} and \cite{KasDmod} for details. We mostly use the notation from the last of these three references.

Let $X$ be a complex manifold with structure sheaf $\cO_X$. We denote by $\cD_X$ the sheaf of (non-commutative) rings of linear partial differential operators with coefficients in $\cO_X$ on $X$. The category of (left) $\cD_X$-modules is denoted by $\Mod{\cD_X}$, and its bounded derived category by $\DbD{X}$. We denote by $\ModcohD{X}$ the full subcategory of \emph{coherent} $\cD_X$-modules, by $\ModholD{X}$ the full subcategory of \emph{holonomic} $\cD_X$-modules, and by $\ModrhD{X}$ the full subcategory of \emph{regular holonomic} $\cD_X$-modules.
We moreover denote by $\DbrhD{X}$ (resp.\ $\DbholD{X}$, $\DbcohD{X}$) the full subcategory of $\DbD{X}$ of complexes with cohomologies in $\ModrhD{X}$ (resp.\ $\ModholD{X}$, $\ModcohD{X}$).

We write $\D_X$ for the duality functor for $\cD_X$-modules, and if $f\colon X\to Y$ is a morphism, we denote by $\DD f_*$ (resp.\ $\DD f^*$) the direct image for $\cD_X$-modules (resp.\ the inverse image for $\cD_Y$-modules) along $f$.

If $D\subset X$ is a hypersurface, we denote by $\cO_X(*D)$ the sheaf of meromorphic functions with poles on $D$ at most. It is a left $\cD_X$-module and for $\cM\in\DbholD{X}$ we denote its \emph{localization at $D$} by $\cM(*D)\vcentcolon= \cM\otimes_{\cO_X} \cO_X(*D)$. We also introduce the notation  $\cM(!D)\defeq \D_X ((\D_X\cM)(*D))$. We call $\cM\in\ModholD{X}$ a \emph{meromorphic connection} along $D$ if $\singsupp(\cM) = D$ and $\cM(*D)\iso\cM$, where $\singsupp(\cM)$ denotes the singular support of $\cM$.

An important basic example of a meromorphic connection which is not regular holonomic is the following: Let $X$ be a complex manifold, $D\subset X$ a normal crossing divisor and $\varphi\in\Gamma(X;\cO_X(*D))$, then the \emph{exponential $\cD_X$-module} $\cE^\varphi$ is defined by $\cE^\varphi\vcentcolon= (\cD_X/\mathrm{ann}(\varphi))(*D)$, where $\mathrm{ann}(\varphi)$ is the (left) ideal of $\cD_X$ of operators annihilating $\varphi$.
	
The classification of holonomic D-modules -- in particular in the case of irregular singularities in higher dimensions -- has been a difficult problem. In the one-dimensional case, it turned out that the Stokes phenomenon is the key ingredient to achieve a complete description of connections with irregular singular points. Very roughly, the Stokes phenomenon in this context is the observation that a meromorphic connection decomposes as a direct sum of certain ``elementary'' connections (namely, exponential D-modules) in sufficiently small sectors with vertex at the singular point, but this decomposition might not exist globally in an open neighbourhood of the singularity. It was conjectured that a similar statement holds for meromorphic connections in higher dimensions, i.e.\ that such local decompositions still exist (as in Definition~\ref{def:normalform} below), at least after suitable blow-ups of the singular locus and de-ramification. A proof was given under certain assumptions in the two-dimensional case in \cite{SabAst}, and finally for the general case by K.\ Kedlaya \cite{Ked1,Ked2} and T.\ Mochizuki \cite{Moc1,Moc2} (see also \cite{Maj} for the asymptotic analysis).

This classification leads to a useful technique for proving statements about holonomic D-modules, and we will briefly recall it here for later use. If $X$ is a complex manifold and $D\subset X$ a normal crossing divisor, we denote by $\XD$ the real oriented blow-up of $X$ along the components of $D$, and we denote by $\Amod$ the sheaf of functions which are holomorphic on $\XD\setminus \partial\XD$ having moderate growth along $\partial \XD$ (see e.g.\ \cite[Notation~7.2.1]{DK16} for a more precise definition, where this sheaf is denoted by $\cA_{\XD}$, or also \cite[§4.1.2]{MocBetti}).
For a $\cD_X$-module $\cM$, we write
$$\cM^{\cA}\vcentcolon= \Amod\otimes_{\varpi^{-1} \cO_X} \varpi^{-1}\cM.$$
Note that $(\cM(*D))^\cA\iso \cM^\cA$ (cf.\ \cite[Lemma~7.2.2]{DK16}).

We give the following definition of a good normal form, combining elements from \cite[Definition 7.3.3]{DK16} and \cite[§2.1]{SabAst} (cf.\ also \cite{Ito}).
\begin{defi}[{cf.\ \cite[Definition 7.3.3]{DK16}}]\label{def:normalform}
	Let $X$ be a complex manifold and $D\subset X$ a normal crossing divisor. A holonomic $\cD_X$-module $\cM\in\ModholD{X}$ is said to have a \emph{good normal form along $D$} if
	\begin{itemize}
		\item $\cM\iso\cM(*D)$,
		\item $\singsupp \cM = D$,
		\item for any $p\in D$, there exists an open neighbourhood $U\subset X$ and a finite good set of functions $\{\varphi_i\mid i\in I\}\subset \Gamma(U;\cO_X(*D))$ such that any $x\in \varpi^{-1}(p)\subset \partial \XD$ has an open neighbourhood $V\subset \varpi^{-1}(U)\subset \XD$ on which there exists an isomorphism
		$$(\cM^\cA)|_V \iso \Big(\bigoplus_{i\in I} \big(\cE^{\varphi_i}\big)^\cA\Big)\Big|_V.$$
		Here, the set $\{\varphi_i\mid i\in I\}$ is called \emph{good} if the divisors associated to the functions $\varphi_i$ and $\varphi_i-\varphi_j$ that are not elements of $\Gamma(U;\cO_X)$, for $i,j\in I$, $i\neq j$, are supported on $D$ and non-positive.
	\end{itemize}
\end{defi}

Using this notion of good normal form, one can deduce from the above-mentioned classification of holonomic D-modules the following fundamental lemma. (This is a slight variation of \cite[Lemma 7.3.7]{DK16}).

\begin{lemma}[{cf.\ \cite[Lemma 7.3.7]{DK16}}]\label{lemma:classif}
	Let $\mathrm{P}_X(\cM)$ be a statement concerning a complex manifold $X$ and an object $\cM\in\DbholD{X}$. Then, $\mathrm{P}_X(\cM)$ is true for any complex manifold $X$ and any $\cM\in \DbholD{X}$ if all of the following conditions are satisfied:
	\begin{itemize}
		\item[(a)] Locality: If $X=\bigcup_{i\in I} U_i$ is an open covering, then $\mathrm{P}_X(\cM)$ is true if and only if $\mathrm{P}_{U_i}(\cM|_{U_i})$ is true for every $i\in I$.
		\item[(b)] Stability by translation: If $n\in \Z$ and $\mathrm{P}_X(\cM)$ is true, then $\mathrm{P}_X(\cM[n])$ is true.
		\item[(c)] Stability in exact triangles: If $\cM'\to\cM\to\cM''\ToPO$ is a distinguished triangle in $\DbholD{X}$ and both $\mathrm{P}_X(\cM')$ and $\mathrm{P}_X(\cM'')$ are true, then $\mathrm{P}_X(\cM)$ is true.
		\item[(d)] Stability by direct summands: If $\cM,\cM'\in\ModholD{X}$ and $\mathrm{P}_X(\cM\oplus \cM')$ is true, then $\mathrm{P}_X(\cM)$ is true.
		\item[(e)] Stability by projective pushforward: If $f\colon X\to Y$ is a projective morphism, $\cM\in\ModholD{X}$ and $\mathrm{P}_X(\cM)$ is true, then $\mathrm{P}_Y(\DD f_*\cM)$ is true.
		\item[(f)] If $\cM\in \ModholD{X}$ has a good normal form along a normal crossing divisor $D\subset X$, then $\mathrm{P}_X(\cM)$ is true.
	\end{itemize}
\end{lemma}

\subsection{Sheaves}
Let $\k$ be a field and let $X$ be a topological space.
We denote by $\Mod{\k_X}$ the category of sheaves of $\k$-vector spaces on $X$ and by $\Dbk{X}$ its bounded derived category. For a $\k$-vector space $V$, denote by $V_X$ the constant sheaf on $X$ with stalk $V$, and an object $\cF\in\Mod{\k_X}$ is called a \emph{local system} (of $\k$-vector spaces) if it is locally isomorphic to a constant sheaf.

If $X$ is a real analytic (resp.\ complex) manifold, we have the full subcategory $\ModRc{\k_X}$ (resp.\ $\ModCc{\k_X}$) of $\Mod{\k_X}$ consisting of $\R$-constructible (resp.\ $\C$-constructible) sheaves and the full subcategory $\DbRck{X}$ (resp.\ $\DbCck{X}$) of $\Dbk{X}$ consisting of complexes with $\R$-constructible (resp.\ $\C$-constructible) cohomologies.

The six Grothendieck operations for sheaves are denoted by $\RR\Hom$, $\otimes$, $\RR f_*$, $f^{-1}$, $\RR f_!$ and $f^!$ for $f\colon X\to Y$ a morphism.

For details on the theory of sheaves of vector spaces and constructibility, we refer to the standard literature, such as \cite{KS90} or \cite{Dimca}.

\subsection{Ind-sheaves and subanalytic sheaves}
Let $X$ be a good topological space.
In \cite{KS01}, M.\ Kashiwara and P.\ Schapira introduced and studied the category  $\mathrm{I}(\k_X)$ of \emph{ind-sheaves} on $X$ as a generalization of the category $\Mod{\k_X}$, and its bounded derived category $\DbIk{X}$. There is a fully faithful and exact embedding $\iota\colon \Mod{\k_X}\hookrightarrow \mathrm{I}(\k_X)$. Since this embedding does not commute with inductive limits in general but the functor $\iota$ is sometimes suppressed in the notation, one uses the notation $\indlim{}$\vspace{-0.25cm} (instead of just $\varinjlim$) for the inductive limit in the category of ind-sheaves. The functor $\iota$ has an exact left adjoint $\alpha\colon \mathrm{I}(\k_X)\to \Mod{\k_X}$, and this functor admits a fully faithful and exact left adjoint giving another embedding $\beta\colon \Mod{\k_X}\hookrightarrow \mathrm{I}(\k_X)$.

It was also shown that if $X$ is a real analytic manifold, the category of ``subanalytic ind-sheaves'' (called ``ind-$\R$-constructible ind-sheaves'' in \cite{KS01}) is equivalent to the category of sheaves on the (relatively compact) subanalytic site $\Xsac$ (see \cite[Theorem 6.3.5]{KS01} and cf.\ \cite{Pre} for a more detailed study of subanalytic sheaves).

Let us briefly recall this version of the subanalytic site: Open sets of $\Xsac$ are relatively compact open subanalytic subsets of $X$. A covering of an open subset $U$ in $\Xsac$ is a covering $U=\bigcup_{i\in I} U_i$ by open sets $U_i$ in $\Xsac$ admitting a finite subcover. For a more detailed study of subanalytic sets, we refer to \cite{BM88} 	and the references therein.
We denote the subcategory of $\mathrm{I}(\k_X)$ consisting of subanalytic ind-sheaves by $\Isuban{\k_X}$.

In some regards, subanalytic sheaves behave differently from sheaves on a usual topology. For example, if we are given a filtrant inductive system $F_j\in\Isuban{\k_X}$, $j\in J$, and $U\in \Xsac$, we have
$$\Gamma(U;\indlim{j\in J} F_j)\iso \varinjlim\limits_{j\in J} \Gamma(U;F_j).$$

The following lemma is also easily proved using the finiteness of the gluing procedure for subanalytic sheaves (see \cite[Proposition~6.4.1]{KS01}).
\begin{lemma}\label{lemma:latticeSectionsSuban}
	Let $F\in \Isuban{\k_X}$ and let $A$ be a $\k$-algebra. Then for $U\subseteq X$ a relatively compact subanalytic open subset, we have $\Gamma(U;A_X\otimes_{\k_X} F)\iso A\otimes_{\k} \Gamma(U;F)$.
\end{lemma}

Some important examples of ind-sheaves are the following: On a real analytic manifold $X$, one has the ind-sheaf $\Cinft{X}$ of tempered complex-valued smooth functions (i.e.\ smooth functions with moderate growth near the boundary of their domain) and the ind-sheaf $\Dbt{X}$ of tempered complex-valued distributions (i.e.\ distributions extending to the whole space). If $X$ is a complex manifold, it is in particular a real analytic manifold, and these two ind-sheaves are modules over $\beta\cD_X$ and $\beta\cD_{\Xc}$. Their Dolbeault complexes are isomorphic and one defines them to be the complex of tempered holomorphic functions
\begin{equation}\label{eq:DefOt}
\cO^{\mathrm{t}}_X\vcentcolon= \RIhom_{\beta\cD_{\Xc}}(\beta\cO_{\Xc},\Cinft{X})\iso \RIhom_{\beta\cD_{\Xc}}(\beta\cO_{\Xc},\Dbt{X})
\end{equation}
(see \cite[§7.2]{KS01}).

\subsection{Enhanced ind-sheaves}\label{sec:enhanced}
In \cite{DK16} and \cite{DK19}, A.\ D'Agnolo and M.\ Kashiwara extended the theory further, introducing and studying the category of \emph{enhanced ind-sheaves} on a so-called \emph{bordered space}. We recall very few basics here and refer to loc.~cit.\ for more details (see also \cite{KS16} for an exposition).

A bordered space $\cX=(X,\widehat{X})$ is a pair of good topological spaces such that $X\subseteq \widehat{X}$ is an open subspace. In particular, every good topological space $X$ can be considered a bordered space $(X,X)$. A morphism $f\colon (X,\widehat{X})\to(Y,\widehat{Y})$ is a continuous map $X\to Y$ such that the morphism $\overline{\Gamma}_f\to \widehat{X}$ induced by the projection to the first factor is proper. Here, $\overline{\Gamma}_f$ denotes the closure in $\widehat{X}\times\widehat{Y}$ of the graph $\Gamma_f\subset X\times Y$. A morphism is called \emph{semi-proper} if also the morphism $\overline{\Gamma}_f\to\widehat{Y}$ induced by the second projection is proper.

We say that $\cX$ is a \emph{real analytic bordered space} if $\widehat{X}$ is a real analytic manifold and $X\subseteq \widehat{X}$ is a subanalytic open subset. Morphisms of real analytic bordered spaces are additionally required to be subanalytic in the sense that $\Gamma_f$ is a subanalytic subset of $\widehat{X}\times\widehat{Y}$.
(Everything we do for real analytic manifolds or bordered spaces could also be done slightly more generally for subanalytic spaces or bordered spaces, see e.g.\ \cite[Exercise~9.2]{KS90} and \cite[§3.1]{DK19} for these notions.)

For the rest of this subsection, let $\cX=(X,\widehat{X})$ be a real analytic bordered space.

The category $\EbIk{\cX}$ of \emph{enhanced ind-sheaves} on $\cX=(X,\widehat{X})$ is a quotient category of $\DbIk{\widehat{X}\times\overline{\R}}$, where $\overline{\R}=\R\sqcup\{\pm\infty\}$. To be more precise, it is constructed in two steps:
First, for any bordered space $\cY=(Y,\widehat{Y})$, one sets
$$\DbIk{\cY} \defeq \DbIk{\widehat{Y}}/\DbIk{\widehat{Y}\setminus Y}.$$
This type of category comes equipped with six Grothendieck operations. One then defines
$$\EbIk{\cX}\defeq \DbIk{\cX\times\R_\infty}/\pi_\cX^{-1}\DbIk{\cX},$$
where one sets $\R_\infty\vcentcolon=(\R,\overline{\R})$, so $\cX\times\R_\infty=(X\times\R,\widehat{X}\times \overline{\R})$ and $\pi_\cX\colon \cX\times\R_\infty\to \cX$ is given by the projection $X\times\R\to X$.

The category $\EbIk{\cX}$ has six Grothendieck operations denoted by $\RR\Ihom^+$, $\conv$, $\EE f_*$, $\EE f^{-1}$, $\EE f_{!!}$, and $\EE f^!$ (for morphisms $f$ of bordered spaces). One also has a sheaf-valued hom functor $$\RHomE\colon \EbIk{\cX}\times\EbIk{\cX} \to \Dbk{X}$$
and for any $\cF\in\Dbk{X}$ a tensor product functor
$$\EbIk{\cX} \to \EbIk{\cX}, \quad H \mapsto \pi^{-1}\cF \otimes H.$$
Here, $\pi\colon X\times\R\to X$ is the projection (and $\pi^{-1}\cF$ is to be seen as extended by zero to $\widehat{X}\times\overline{\R}$).

The natural t-structure on the derived category $\DbIk{\widehat{X}\times\overline{\R}}$ induces a t-structure on $\EbIk{\cX}$, whose heart is denoted by $\EzIk{\cX}$ (these are therefore the objects represented by complexes of ind-sheaves concentrated in degree $0$).
Moreover, a notion of $\R$-constructibility is defined for enhanced ind-sheaves, leading to the full subcategory $\EbRcIk{\cX}\subset \EbIk{\cX}$. We write $\EzRcIk{\cX}\vcentcolon= \EzIk{\cX}\cap \EbRcIk{\cX}$.

An important object in $\EbIk{\cX}$ is
$$\k^\EE_\cX \vcentcolon= \indlim{a\to\infty} \k_{\{t\geq a\}}\in\EzRcIk{\cX},$$
where $\{t\geq a\}\vcentcolon= \{(x,t)\in \widehat{X}\times\overline{\R}\mid x\in X, t\in\R, t\geq a\}\subset \widehat{X}\times\overline{\R}$.

One has an embedding of the category of sheaves on $X$ into the category of enhanced ind-sheaves on $\cX$
$$\e{\k}{\cX}\colon \Dbk{X} \hookrightarrow \EbIk{\cX}, \quad \cF\mapsto \pi^{-1}\cF\otimes \k^\EE_{\cX}.$$
One obtains a natural duality functor for enhanced ind-sheaves, which is given as
$$\DE{\cX}\colon \EbIk{\cX}^\mathrm{op}\to \EbIk{\cX}, \quad H \mapsto \RIhom^+(H,\e{\k}{\cX}(\omega_X)),$$
where $\omega_X\in\Dbk{X}$ is the dualizing complex of $X$.

A fundamental class of objects is that of so-called \emph{exponential enhanced ind-sheaves}:
If $W\subseteq X$ is open such that $W\subseteq \widehat{X}$ is subanalytic, and if $f\colon W\to \R$ is a continuous subanalytic function, we set
$$\E^f_{W|\cX,\k} \vcentcolon= \indlim{a\to\infty} \k_{\{t\geq -f + a\}}\iso \k^\EE_{\cX} \conv \k_{\{t=-f\}}\in\EzRcIk{\cX},$$
where $\{t\geq -f+a\}\vcentcolon= \{(x,t)\in \widehat{X}\times \overline{\R} \mid x\in W, t\in \R, t \geq -f(x)+a \}\subset \widehat{X}\times\overline{\R}$, and similarly one defines $\{t=-f\}\subset \widehat{X}\times\overline{\R}$.
If $\k=\C$ or if there is no confusion about the field of coefficients, we often omit the subscript $\k$ and just write $\E^f_{W|\cX}$ for this object.

\begin{rem}\label{rem:EfSuban}
	For simplicity, we assume $\widehat{X}=X$.
	The category $\EbIk{X}$ can be viewed (via the left adjoints of the quotient functors) as a full subcategory of $\DbIk{X\times\overline{\R}}$. In particular, the image of $\EzRcIk{X}$ lies in $\Isuban{\k_{X\times\overline{\R}}}$. Hence, the object $\E^f_{W|X,\k}$ can be understood as a sheaf on the subanalytic site $(X\times\overline{\R})^c_{\sa}$.
	
	If $X$ is a real analytic manifold, $W\subset X$ an open subanalytic subset and $f\colon W\to \R$ a continuous subanalytic function, then $\E^f_{W|X}=\E^f_{W|X,\k}$ is described as follows (we do not write the index $\k$ here for better readability):
	
	Let $U\in \mathrm{Op}((X\times\overline{\R})^c_{\sa})$.
	\begin{itemize}
		\item If $U\subseteq W\times\R$,
		$$\E^f_{W|X}(U)\iso \varinjlim_{a\to\infty} \Gamma(U;\k_{\{t\geq -f+a\}})\iso \varinjlim_{a\to\infty} \k^{\pi_0(U\cap\{t\geq- f+a\})},$$
		where $\pi_0(U\cap\{t\geq-f+a\})$ is the set of connected components of the intersection $U\cap \{t\geq -f+a\}$ (note that the number of connected components is finite since this intersection is subanalytic and relatively compact).
		\item More generally, if $U\subseteq X\times\overline{\R}$, then
		$$\E^f_{W|X}(U)\iso \varinjlim_{a\to\infty} \k^{\pi_0^{<\infty}(U\cap\{t\geq-f+a\})},$$
		where $\pi_0^{<\infty}(U\cap\{t\geq-f+a\})$ is the set of connected components $Y$ of the intersection $U\cap \{t\geq -f+a\}$ such that $\overline{Y}\cap U$ does not intersect $X\times\{+\infty\}$ and $(X\setminus W)\times \overline{\R}$.
	\end{itemize}
	In particular, we have $\E^f_{W|X}(W\times\R)=\k$ and $\E^f_{W|X}(S\times\overline{\R})=0$ for $S\subseteq X$ open, as well as $\E^f_{W|X}(X\times I)=0$ for $I\subseteq \overline{\R}$ open if $W\neq X$.
	For $V\subseteq U$, the restriction maps $\E^f_{W|X}(U)\to\E^f_{W|X}(V)$ are given in the obvious way, by a combination of identities, zero maps or diagonal maps $\k\to \k^M$ (for $M$ a set).
	
	The sheaves $\E^f_{W|X}$ are therefore very similar to sheaves of the form $\k_Z$ on the usual topology, i.e.\ constant sheaves on a closed subset $Z\subseteq X$, extended by zero outside $Z$. For $\E^f_{W|X}$, the set $Z$ is to be thought of as a half-space lying over the graph of some real-valued function and considered to be shifted towards ``infinitely high real values''.
	
	A typical example is the case where $X$ is a disc around a point $p$ in the complex plane, $W=X\setminus\{p\}$ or $W=S$ an open sector with vertex $p$, and $f=\real \varphi$ is the real part of a holomorphic function $\varphi\colon W\to \C$ with a pole at $p$.
\end{rem}

We recall the following definition (note that the conventions are different between \cite{MocCurveTest2} and \cite{DKmicrolocal}, and we use the latter one).
\begin{defi}[{cf.\ e.g.\ \cite[Notation~3.2.1]{DKmicrolocal}}]
	Let $\cX$ be a bordered space. Let $W\subseteq X$ be an open subset and let $f,g\colon W\to\R$ be continuous functions. Then we write
	\begin{align*}
		f\sim_W g \;\; \vcentcolon\Leftrightarrow\;\; &f-g \text{ is bounded on any $V\subseteq W$, $V\subset\subset \widehat{X}$},\\
		f\preceq_W g \;\; \vcentcolon\Leftrightarrow\;\;  &f-g \text{ is bounded from above on any $V\subseteq W$, $V\subset\subset \widehat{X}$},\\
		f\prec_W g \;\; \vcentcolon\Leftrightarrow\;\;  &f\preceq_W g \text{ but not } f\sim_W g,\\
		&\text{i.e.\ $f-g$ bounded above on each $V$ but unbounded below on some $V$.}
	\end{align*}
	Here, $V\subset\subset\widehat{X}$ means ``$V$ is relatively compact in $\widehat{X}$''.
	
	If $X$ is a complex manifold and $\varphi,\psi\colon W\to \C$ are holomorphic, we will write for short $\varphi\preceq_W\psi$ instead of $\real\varphi \preceq_W \real\psi$ (and similarly for $\sim_W$ and $\prec_W$).
\end{defi}

Note that if $f\sim_{W}g$, then $\E^f_{W|\cX}\iso \E^g_{W|\cX}$.

To end this subsection, let us also recall the following fundamental property (cf.\ e.g.\ \cite[Lemma~9.8.1]{DK16} or \cite[Lemma~3.2.2]{DKmicrolocal}). It can directly be proved using \cite[Proposition~4.7.9]{DK16}.
\begin{lemma}\label{lemma:homSumE}
	Let $\cX$ be a bordered space, let $W\subseteq X$ be open and relatively compact in $\widehat{X}$, and let $f_1,\ldots,f_n, g_1,\ldots, g_m\colon W\to\R$ be continuous subanalytic functions. Then
	\begin{align*}
	\mathrm{Hom}_{\EbIk{\cX}}&\Big(\bigoplus_{k=1}^n \E^{f_k}_{W|\cX}, \bigoplus_{j=1}^m \E^{g_j}_{W|\cX} \Big)\\
	 &\iso  \Big\{ A=(a_{jk})\in \k^{m\times n} \Big| a_{jk}=0 \text{ if } f_k \prec_{W'} g_j \text{ for some open } W'\subseteq W \Big\}.
	 \end{align*}
\end{lemma}
Morphisms between direct sums of exponentials can therefore -- after numbering the direct summands -- be represented by matrices (and composition corresponds to matrix multiplication). In particular, if $f_1\prec_W f_2\prec_W \ldots \prec_W f_n$, then endomorphisms of $\bigoplus_{i=1}^n \E^{f_i}_{W|\cX}$ are represented by upper-triangular square matrices.

\subsection{Riemann--Hilbert correspondences for analytic D-modules}
It is a classical idea to ask if the functor associating to a differential system its solution space is an equivalence. (The question about surjectivity of such a functor for Fuchsian differential equations goes back at least to Hilbert's 21st problem.)

Let $X$ be a complex manifold of (complex) dimension $d_X$.
Classically, in the theory of D-modules, one studies the following objects:
Let $\cM\in\DbholD{X}$. Denote by $\Omega_X$ the sheaf of top-degree holomorphic differential forms on $X$. Then one defines the \emph{holomorphic de Rham complex} of $\cM$ by
$$\DR{X}(\cM)\vcentcolon= \Omega_X\LtensD{X} \cM$$
and the \emph{holomorphic solution complex} of $\cM$ by
$$\Sol{X}(\cM)\vcentcolon= \RR\Hom_{\cD_X}(\cM,\cO_X).$$

In \cite{KasRHreg}, M.\ Kashiwara proved that the de Rham functor gives an equivalence
$$\DR{X}\colon \DbrhD{X}\overset{\sim}{\To} \DbCcC{X}.$$

It was then natural to search for a generalization of this result to holonomic D-modules (relaxing the regularity assumption). It was clear that the functor $\DR{X}$ is no longer fully faithful on the category $\DbholD{X}$, and hence the framework (the functor and the target category) would need to be modified.

Finally, A.\ D'Agnolo and M.\ Kashiwara (see \cite{DK16}) were able to establish a fully faithful functor, the \emph{enhanced de Rham functor}
$$\DRE{X}\colon \DbholD{X}\lhook\joinrel\To \EbRcIC{X}.$$

Let us briefly give some details on its construction:

Starting from the ind-sheaf of tempered holomorphic functions $\cO^\mathrm{t}_X\in \DbIC{X}$ that had been introduced in \cite{KS01}, the authors of \cite{DK16} define the enhanced ind-sheaf of enhanced tempered holomorphic functions $\OE{X}\in \EE^\mathrm{b}(\mathrm{I}\C_X)$. This is done by first defining enhanced tempered distributions $\DbE$ and setting $\OE{X}\vcentcolon=\Omega_{\Xc}\otimes_{\cD_{\Xc}} \DbE$, in analogy to \eqref{eq:DefOt}. (Recall that $\Xc$ denotes the complex conjugate manifold.)
They then set $\Omega^\EE_X\vcentcolon= \Omega_X \LtensO{X} \OE{X}$ and define the functors
$$\DRE{X}(\cM) \vcentcolon= \Omega^\EE_X\LtensD{X} \cM$$
and
$$\SolE{X}(\cM) \vcentcolon= \RHom_{\cD_X}(\cM,\OE{X}).$$

These two functors are related by duality (recall the duality functors for D-modules and enhanced ind-sheaves from Sections~\ref{sec:Holonomic} and \ref{sec:enhanced}, respectively):
$$\DE{X}\DRE{X}(\cM)\iso\DRE{X}(\D_X\cM)\iso \SolE{X}(\cM)[d_X].$$

For exponential D-modules, the enhanced solution and de Rham functors are explicitly described as follows: Consider a normal crossing divisor $D\subset X$ and a function $\varphi\in\Gamma(X;\cO_X(*D))$, then
$$\SolE{X}(\cE^\varphi) \iso \E^{\real\varphi}_{X\setminus D|X},\qquad \DRE{X}(\cE^\varphi)\iso \RIhom(\pi^{-1}\C_{X\setminus D},\E^{-\real \varphi}_{X\setminus D|X})[d_X]$$
(see \cite[Corollary~9.4.12 and Lemma~9.3.1]{DK16}).

\begin{rem}\label{rem:beta}
	The framework and notation are obviously set up to be similar to the regular case, and we refer to \cite{DK16} and \cite{KS16} for more details. Let us, however, remark that the notation is not self-explanatory here, and quite some simplification in notation has happened between \cite{KS01} and works like \cite{DK16} and \cite{KS16}. In particular, whenever a sheaf (as opposed to an enhanced ind-sheaf) appears in the above formulae, it should be read as $\beta\pi^{-1}$ of this sheaf to make sense of these expressions.
	For example, in the notation of \cite{KS01}, the definition of the enhanced solutions functor reads as
	$$\SolE{X}(\cM)=\RR\Ihom_{\beta\pi^{-1}\cD_X}(\beta\pi^{-1}\cM,\OE{X}).$$
\end{rem}

\subsection{Galois conjugation}\label{subsec:GalConj}
Let $L/K$ be a field extension and let $g\in \mathrm{Aut}(L/K)$, i.e.\ $g$ is a field automorphism of $L$ with $g|_K=\id_K$. Then, for an $L$-vector space $V$ we can define its $g$-conjugate $L$-vector space $\overline{V}^g$, which has the same underlying abelian group, but the action of $L$ is given by $\ell\cdot v\vcentcolon= g(\ell)v$ for $\ell\in L$, $v\in V$. This defines a functor from the category of vector spaces to itself (which sends a morphism to the same set-theoretic map).

This conjugation naturally induces a $g$-conjugation on sheaves of $L$-vector spaces, as well as on ind-sheaves over $L$ and enhanced ind-sheaves over $L$, i.e.\ we have functors
$$\overline{(\bullet)}^g\colon \DbL{X}\to \DbL{X}, \qquad \overline{(\bullet)}^g\colon \EbIL{\cX}\to \EbIL{\cX}$$
for a topological space $X$ (resp.\ a bordered space $\cX$), and these functors also preserve $\R$- and $\C$-constructibility.

In the case $L=\C$, $K=\R$, there is only one nontrivial element $\gamma\in \mathrm{Aut}(\C/\R)$, which is complex conjugation, and we therefore simply write $\overline{(\bullet)}$ instead of $\overline{(\bullet)}^\gamma$.

If $X$ is a complex manifold and $\Xc$ its complex conjugate, their underlying topological spaces are the same and hence the categories $\Dbk{X}$ and $\Dbk{\Xc}$ (for a field $\k$) are naturally identified, and similarly for $\EbIk{X}$ and the associated subcategories of $\R$- and $\C$-constructible objects. We can identify $\cO_{\Xc}=\overline{\cO_X}$ and $\cD_{\Xc}=\overline{\cD_X}$, and for a $\cD_X$-module $\cM$, the conjugate $\overline{\cM}$ is naturally a $\cD_{\Xc}$-module. In this situation, it is not difficult to see that
$$\DR{\Xc}(\overline{\cM})=\overline{\DR{X}(\cM)}, \qquad \DRE{\Xc}(\overline{\cM})= \overline{\DRE{X}(\cM)}.$$
	
\section{Conjugation of D-modules and the enhanced de Rham functor}\label{sec:KashiwaraConj}

In this section, our aim is to prove an analogue of a result of \cite{KasConj} in the context of the enhanced de Rham functor. Let us briefly recall the idea and argument of the main statement in \cite{KasConj}.

The Riemann--Hilbert correspondence for regular holonomic D-modules states that the de Rham functor induces an equivalence of categories
$$\DR{X} \colon \DbrhD{X} \overset{\sim}{\longrightarrow} \DbCcC{X}.$$
On the right-hand side, complex conjugation defines an auto-equivalence. It is therefore natural to ask what the corresponding operation on the left-hand side is. In other words, the question is the following: Given an object $\cM\in\DbrhD{X}$, how is the object $\cN\in\DbrhD{X}$ satisfying $\DR{X}(\cN)\iso \overline{\DR{X}(\cM)}$ related to $\cM$?

Indeed, M.\ Kashiwara is able to define a functor $c\colon \DbD{X}\to \DbD{X}$ such that the desired description is $\cN=c(\cM)$, i.e.\
\begin{equation}\label{eq:DRconj}
	\DR{X}(c(\cM)) \iso \overline{\DR{X}(\cM)}
\end{equation}
for $\cM\in\DbrhD{X}$ (in fact, $\cM\in\DbcohD{X}$ is enough here).

The key to this result is an intermediate step, namely the definition and study of a functor
\begin{align*}
C_X\colon \DbD{X}&\to \DbD{\Xc},\\
\cM &\mapsto \RHom_{\cD_X}(\cM,\Db{X}),
\end{align*}
later baptized the \emph{Hermitian duality functor}, where $\Db{X}$ denotes the sheaf of distributions.

One of the key observations in the proof of formula \eqref{eq:DRconj} is the fact that one has an isomorphism of functors\footnote{Comparing this directly with what is written in \cite{KasConj}, our formula here differs by a shift. This is because we use a different convention for the de Rham functor (which actually seems to be more common nowadays and is consistent with later works like \cite{DK16} and \cite{KS16}, for example).} (note that $\DbCcC{X}$ and $\DbCcC{\Xc}$ are naturally identified)
\begin{equation}\label{eq:DRCSol}
	\DR{\Xc}\circ C_X[-d_X] \iso \Sol{X}
\end{equation}
on the category $\DbrhD{X}$ (indeed on $\DbcohD{X}$), and this follows directly from the fact that the Dolbeault complex with distribution coefficients is a resolution of the sheaf of holomorphic functions, i.e.\ $\DR{\Xc}(\Db{X})\iso \cO_X[d_X]$.

In \cite{KasConj}, it was in particular shown that $C_X$ (and hence $c$) preserves regular holonomicity and induces an equivalence of categories if we restrict to regular holonomic D-modules. Moreover, it was a conjecture of M.~Kashiwara (see \cite[Remark~3.5]{KasConj}) that this is also true if we erase the word ``regular'' and just restrict to holonomic D-modules. This conjecture has been proved by C.\ Sabbah and T.\ Mochizuki (\cite[Theorem~3.1.2]{SabAst}, \cite[Corollary~4.19]{MocStokesMero}, see also \cite[§12.6]{SabStokes}).

The idea in this section is now to establish a statement similar to \eqref{eq:DRconj} for the enhanced de Rham functor. More precisely, the idea is that the same relation holds if we replace $\DR{X}$ by $\DRE{X}$ and consider $\cM\in\DbholD{X}$. The functor $c$ will remain unchanged. Due to the results in the holonomic context mentioned above, such a compatibility might not be surprising.
On the other hand, the argument for a statement like \eqref{eq:DRCSol} for the enhanced de Rham and solution functors seems not as direct as in the classical case: The expression that appears is $\DRE{\Xc}(\Db{X})$, while the definition of enhanced tempered holomorphic functions is rather ``$\DR{\Xc}(\cD b^\EE_{X})\iso \OE{X}[d_X]$''.

We therefore apply a different method of proof here, using the classification of general holonomic D-modules due to C.\ Sabbah, K.\ Kedlaya and T.\ Mochizuki (see Section~\ref{sec:Holonomic}), to establish the generalization of \eqref{eq:DRCSol}. This is mainly achieved by combining results about Hermitian duality that have already been proved earlier and that we will briefly review in the next subsection.

\subsection{The Hermitian dual}
Let $X$ be a complex manifold.
In \cite{KasConj}, M. Kashiwara introduced the functor $C_X\colon \DbD{X}\to \DbD{\Xc}$ given by
$$C_X(\cM)\vcentcolon= \RHom_{\cD_X}(\cM,\Db{X}).$$
Here, $\Db{X}$ is the sheaf of Schwartz distributions on the underlying real manifold of $X$. It is a module over $\cD_X$ and $\cD_{\Xc}$.

This functor has been further studied by other authors, in particular by C.\ Sabbah and T.\ Mochizuki, and we recall two important properties that we will use below.

\begin{lemma}[{see \cite[§12.5.2]{MocMT}}] \label{lemma:CXpushforward}
	Let $f\colon X\to Y$ be a morphism of complex manifolds and let $\cM$ be a holonomic $\cD_X$-module such that $f$ is proper on $\supp \cM$. Then there is a natural isomorphism
	\begin{equation*}
		\DD \overline{f}_* C_X(\cM) \iso C_Y(\DD f_*\cM).
	\end{equation*}
\end{lemma}

For the next lemma, recall the notations $\cM(*D)$ and $\cM(!D)$ from Section~\ref{sec:Holonomic}. Moreover, note that if $X$ is a complex manifold and $D\subset X$ is a hypersurface, this also defines a hypersurface $\overline{D}\subset \overline{X}$ in the complex conjugate manifold.

\begin{lemma}\label{lemma:ConjLoc}
	Let $\cM$ be a holonomic $\cD_X$-module and $D\subset X$ a hypersurface. Then there is an isomorphism
	$$C_X(\cM(*D)) \iso C_X(\cM)(!\overline{D}).$$
\end{lemma}
\begin{proof}
	Let us set $\cN=\overline{C_X(\cM)}$ and let us denote by $C\colon \cM\times\overline{\cN} \to \Db{X}$ the canonical pairing. Then $(\cM,\cN,C)$ is a non-degenerate $\cD$-triple (in the sense of \cite{MocMT}, for example). Then by \cite[§12.2.2]{MocMT}, the $\cD$-triple $(\cM(*D), \cN(!D),C(!D))$ is still non-degenerate. In particular, $C_X(\cM(*D))\iso \overline{\cN(!D)}\iso C_X(\cM)(!\overline{D})$, as desired.
\end{proof}

Recall the notation on the real oriented blow-up from Section~\ref{sec:Holonomic}. If $X$ is a complex manifold and $D\subset X$ is a normal crossing divisor, we also have the normal crossing divisor $\overline{D}\subset \Xc$ (if $D$ is locally given by $\{z_1\cdot\ldots\cdot z_k=0\}$, then $\overline{D}$ is locally given by $\{\overline{z_1}\cdot\ldots\cdot \overline{z_k}=0\}$) and one has the real blow-up space $\overline{\varpi}\colon \widetilde{\Xc}\to \Xc$. (Its underlying topological space is naturally identified with that of $\widetilde{X}$ and $\cA_{\widetilde{\Xc}}^{\mathrm{mod}}=\overline{\Amod}$. We will often just write $\varpi$ instead of $\overline{\varpi}$.)
If $\cM$ is a $\cD_{\Xc}$-module, we write
$$\cM^{\overline{\cA}}\vcentcolon= \cA_{\widetilde{\Xc}}^{\mathrm{mod}}\otimes_{\varpi^{-1} \cO_{\Xc}} \varpi^{-1}\cM$$
to emphasize that we are taking the tensor product with $\cA_{\widetilde{\Xc}}^\mathrm{mod}$, the sheaf of \emph{antiholomorphic} functions with moderate growth at $\partial \widetilde{\Xc}$.

\begin{lemma}\label{lemma:CXEphi}
	Let $X$ be a complex manifold, $D\subset X$ a normal crossing divisor and $\varphi\in \Gamma(X;\cO_X(*D))$. Then there is an isomorphism
	$$(C_X(\cE^\varphi))^{\overline{\cA}}\iso (\cE^{-\overline{\varphi}})^{\overline{\cA}}.$$
\end{lemma}
\begin{proof}
	We use the notation in \cite{SabAst} without explaining all the details, in particular the functor $C_X^{\mathrm{mod}\, D}=\RHom_{\cD_{X}}(-,\mathcal{D}b^{\mathrm{mod}\, D}_{X})$. There are isomorphisms
	\begin{align*}
		(C_X(\cE^\varphi))^{\overline{\cA}}&\iso \big(C_X(\cE^\varphi)(*\overline{D})\big)^{\overline{\cA}} \iso \big(C_X(\cE^\varphi(!D))\big)^{\overline{\cA}}\iso (C_X^{\mathrm{mod}\, D}(\cE^\varphi))^{\overline{\cA}}
	\end{align*}
	where the second isomorphism follows from Lemma~\ref{lemma:ConjLoc} and the last isomorphism follows from \cite[Proposition~3.2.2]{SabAst}.
	 On the other hand, we know from \cite[p.\ 69]{SabAst} that
	 \begin{align*}
	 	(C_X^{\mathrm{mod}\, D}(\cE^\varphi))^{\overline{\cA}} \iso C_{\widetilde{X}}^{\mathrm{mod}\, D}\big((\cE^\varphi)_{\widetilde{X}}\big)\iso (\cE^{-\overline{\varphi}})^{\overline{\cA}}.
	 \end{align*}
\end{proof}

The following result is shown in \cite[\S II.3]{SabAst} (see also \cite[\S 12.6]{SabStokes}). 
It states that a decomposition as in the definition of good normal form induces a similar decomposition for the Hermitian dual, and describes how its exponential factors change under Hermitian duality.

\begin{prop}[{\cite{SabAst}}]
	\label{prop:AnStrCX}
	Let $D\subset X$ be a normal crossing divisor and $\cM$ a holonomic $\cD_X$-module with a good normal form along $D$. If $p\in D$ and $\{\varphi_i\mid i\in I\}$ is a good set of functions in a neighbourhood of $p$ as in Definition~\ref{def:normalform} such that every point of $\varpi^{-1}(p)$ has an open neighbourhood $V\subset \widetilde{X}$ with an isomorphism
	$$(\cM^\cA)|_V \iso \Big(\bigoplus_{i\in I} \big(\cE^{\varphi_i}\big)^\cA\Big)\Big|_V,$$
	then every point of $\overline{\varpi}^{-1}(p)$ has an open neighbourhood $V'\subset \widetilde{\overline{X}}$ with an isomorphism
	$$\big(C_X(\cM)^{\overline{\cA}}\big)\big|_{V'} \iso \Big(\bigoplus_{i\in I} \big(\cE^{-\overline{\varphi_i}}\big)^{\overline{\cA}}\Big)\Big|_{V'}.$$
\end{prop}

\subsection{Hermitian duality and the enhanced de Rham complex}

As mentioned above, the main step in proving the desired generalization of \eqref{eq:DRconj} is the proof of an analogue of the isomorphism \eqref{eq:DRCSol} in the setting of the enhanced de Rham and solution functors. This will be done in Proposition~\ref{prop:DRECSolEIso} below. Let us first show the existence of a canonical morphism.

\begin{lemma}\label{lemma:morphDRECSolE}
	Let $X$ be a complex manifold and $\cM\in \DbcohD{X}$. There is a canonical morphism, functorial in $\cM$,
	\begin{equation}\label{eq:canMorphCX}
		\DRE{\Xc}(C_X(\cM))[-d_X] \To \SolE{X}(\cM).
	\end{equation}
\end{lemma}

In the proof of this statement, we will use the following lemma. We write $\sP\vcentcolon= \PP^1(\R)$ and $\PP\vcentcolon=\PP^1(\C)$ for the real and the complex projective line, respectively.

\begin{lemma}\label{lemma:DbtXP}
	Let $k\colon X\times\sP\to X\times\PP$ be the inclusion. Then there is an isomorphism
	$$\Dbt{X\times\sP} \iso k^!\RHom_{\cD_{\PP}}(\cO_{\PP},\Dbt{X\times\PP})[1].$$
\end{lemma}
\begin{proof}
	First, we note that for a right $\cD_{\PP}$-module $\cN$ we have an isomorphism
	$$k^!(\cN\Ltens{\cD_{\PP}} \Dbt{X\times\PP})\iso \DD k^*\cN\Ltens{\cD_{\sP}} \Dbt{X\times\sP},$$
	which follows from \cite[Lemma~5.3.2]{DK16} (which is itself derived from results of \cite{KS96}).
	
	Using this (in the second isomorphism below), one obtains
	\begin{align*}
		k^!\RHom_{\cD_{\PP}}(\cO_{\PP},\Dbt{X\times\PP})[1]
		&\iso k^!\big( \Omega_{\PP} \Ltens{\cD_{\PP}}\Dbt{X\times\PP}\big)\\
		&\iso \DD k^*\Omega_{\PP} \Ltens{\cD_{\sP}}\Dbt{X\times\sP}\\
		&\iso \Dbt{X\times\sP}.
	\end{align*}
	Here, the last isomorphism follows from \cite[Lemma~2.3.2]{KasMThesis}.
\end{proof}

\begin{proof}[Proof of Lemma~\ref{lemma:morphDRECSolE}]
	We have 
	\begin{align*}
		\DRE{\Xc}(C_X(\cM)) &\iso \Omega^\EE_{\Xc} \Ltens{\cD_{\Xc}} \RHom_{\cD_X}(\cM,\Db{X})\\
		&\iso \RHom_{\cD_X}(\cM,\Omega^\EE_{\Xc} \Ltens{\cD_{\Xc}}\Db{X})
	\end{align*}
and
\begin{align*}
	\SolE{X}(\cM) \iso \RHom_{\cD_X}(\cM,\OE{X}).
\end{align*}
Therefore, it suffices to find a canonical morphism $\Omega^\EE_{\Xc} \Ltens{\cD_{\Xc}}\Db{X}[-d_X]\to \OE{X}$. Note further that 
\begin{align*}
	\Omega^\EE_{\Xc} \Ltens{\beta\cD_{\Xc}}\Db{X} &\iso (\Omega_{\Xc} \Ltens{\cO_{\Xc}} \OE{\Xc})\Ltens{\cD_{\Xc}} \Db{X}\\ 
	&\iso \Omega_{\Xc} \Ltens{\cD_{\Xc}} (\OE{\Xc} \Ltens{\cO_{\Xc}} \Db{X})
\end{align*}
and
\begin{align*}
	\OE{X} &\iso \Omega_{\Xc} \Ltens{\cD_{\Xc}} \DbE[-d_X],
\end{align*}
which means that it suffices to find a morphism
$$\OE{\Xc} \Ltens{\cO_{\Xc}} \Db{X} \to \DbE.$$

Let us consider the natural morphisms of bordered spaces
$$\begin{tikzcd}
	X\times \R_\infty \arrow{r}{j} & X\times\sP\arrow{r}{k} & X\times\PP.
\end{tikzcd}$$
One has
\begin{align*}
	\OE{\Xc} \Ltens{\cO_{\Xc}} \Db{X} \iso j^!\RHom_{\cD_\sP}(\cE^t,k^!\cO^\mathrm{t}_{\Xc\times \PP}\Ltens{\cO_X} \Db{X} )[2]
\end{align*} 
and $$\DbE \iso  j^!\RHom_{\cD_\sP}(\cE^t,\Dbt{X\times\sP})[1],$$
so it is enough to construct a morphism (for any complex manifold $X$)
$$k^!\cO^\mathrm{t}_{X\times \PP}\Ltens{\cO_X} \Db{X} [1] \To \Dbt{X\times\sP}.$$

Now we note that we have a morphism
\begin{align*}
k^!\cO^\mathrm{t}_{X\times \PP} &\iso k^!\RHom_{\cD_{\overline{X\times\PP}}}(\cO_{\overline{X\times\PP}},\Cinft{X\times\PP})\\
&\To k^!\RHom_{\cD_{\overline{X\times\PP}}}(\cD_{\Xc}\overset{\DD}{\boxtimes}\cO_{\overline{\PP}},\Cinft{X\times\PP})\\
&\iso  k^!\RHom_{\cD_{\overline{\PP}}}(\cO_{\overline{\PP}},\Cinft{X\times\PP})
\end{align*}
and that
$$\Dbt{X\times\sP} \iso k^!\RHom_{\cD_{\overline{\PP}}}(\cO_{\overline{\PP}},\Dbt{X\times\PP})[1]$$
by Lemma~\ref{lemma:DbtXP} (observing that we can clearly replace $\PP$ by $\overline{\PP}$ there).

Hence, all we finally need to construct is a morphism
$$\Cinft{X\times\PP}\LtensO{X} \Db{X}\to \Dbt{X\times\PP}.$$

Using a slightly more precise notation (explicitly writing the functor $\beta$ and the pullback via the projection $p\colon X\times\PP\to X$, cf.\ Remark~\ref{rem:beta}), this morphism follows from the natural morphism
\begin{equation}\label{eq:multCDb}
\Cinft{X\times\PP}\Ltens{\beta p^{-1}\cO_{X}} \beta p^{-1}\Db{X} \to \Cinft{X\times\PP}\Ltens{\beta p^{-1}\cO_{X}} \Dbt{X\times\PP} \to \Dbt{X\times\PP}.
\end{equation}

Here, the first arrow is induced by the composition of the natural morphisms $\beta p^{-1}\Db{X}\to p^{-1} \Dbt{X}\to \Dbt{X\times\PP}$ (the second of these arrows being given by taking the product of a distribution on $X$ with the constant distribution $1$ on $\PP$, see e.g.\ \cite[Chap.\ 5]{Hor98}). The second arrow in \eqref{eq:multCDb} is given by multiplication of tempered smooth functions with tempered distributions (see \cite[Lemma~3.3]{KasRHreg} and also \cite[Proposition~5.1.3]{DK16}). More precisely, we use the natural morphism $$\Cinft{X\times\PP}\Ltens{\beta p^{-1}\cO_{X}} \Dbt{X\times\PP}\to \Cinft{X\times\PP}\otimes_{\beta p^{-1}\cO_{X}} \Dbt{X\times\PP}$$ composed with this multiplication.

\end{proof}

The classification of holonomic D-modules now enables us to prove that the morphism from Lemma~\ref{lemma:morphDRECSolE} is an isomorphism for any holonomic D-module $\cM$.

\begin{prop}\label{prop:DRECSolEIso}
	The morphism \eqref{eq:canMorphCX} 
	$$\DRE{\Xc}(C_X(\cM))\To \SolE{X}(\cM)$$
	is an isomorphism for any complex manifold $X$ and any $\cM\in\DbholD{X}$.
\end{prop}

\begin{proof}
	Let us consider the following statement for any complex manifold $X$ and any $\cM\in\DbholD{X}$:
	$$\Prop{\cM}\colon \textnormal{The morphism \eqref{eq:canMorphCX} is an isomorphism.}$$	
	To show that it holds for any $X$ and $\cM$, we will employ Lemma~\ref{lemma:classif}, and	we have to check conditions (a)--(f).
	
	Condition (a) is clear by the compatibility of $\DRE{X}$ and $\SolE{X}$ with inverse images. (Note that the a-priori existence and canonicity of the morphism \eqref{eq:canMorphCX} is essential here.) Condition (b) is also obvious. Condition (c) holds by the axioms of a triangulated category, and condition (d) is true since the morphism \eqref{eq:canMorphCX} is functorial with respect to direct sums.
	
	Condition (e) is proved as follows: Assume that $\Prop{\cM}$ holds, and let $f\colon X\to Y$ be a projective morphism. It also induces a morphism $f\colon \Xc\to \Yc$ (which we will not distinguish in the notation). Then we have
	\begin{align*}
		\SolE{Y}(\DD f_* \cM)&\iso \EE f_* \SolE{X}(\cM)[d_X-d_Y]\\
		&\iso \EE \overline{f}_* \DRE{\Xc}(C_X(\cM))[-d_Y]\\
		&\iso \DRE{\Yc}(\DD \overline{f}_* C_X(\cM))[-d_Y]\\
		&\iso \DRE{\Yc}(C_Y(\DD f_* \cM))[-d_Y],
	\end{align*}
	where the last isomorphism follows from Lemma~\ref{lemma:CXpushforward}.

	It remains to check condition (f): Let $\cM\in\ModholD{X}$ have a good normal form along a normal crossing divisor $D\subset X$ (in particular, $\cM$ is meromorphic at $D$). Without recalling the details (which are not essential for understanding the idea of proof), we use here some notations and results for enhanced de Rham and solution complexes on the real blow-up space (see \cite[§9.2]{DK16}).
	
	Due to Lemma~\ref{lemma:morphDRECSolEbu} below, there is a natural morphism similar to \eqref{eq:canMorphCX} on the real blow-up
	\begin{align}\label{eq:morphDRECSolEbu}
		\DRE{\widetilde{\Xc}}\big(C_X(\cM)^{\overline{\cA}}\big)[-d_X]\to \SolE{\widetilde{X}}(\cM^\cA).
	\end{align}
	
	We want to show that this is an isomorphism.
	Locally on $\widetilde{X}=\widetilde{\Xc}$ (they are identified as real analytic manifolds with corners), $\cM^\cA$ decomposes as a direct sum of exponential D-modules $(\cE^{\varphi})^\cA$, and $C_X(\cM)^{\overline{\cA}}$ decomposes accordingly into exponential D-modules $(\cE^{-\overline{\varphi}})^{\overline{\cA}}$ (cf.\ Proposition~\ref{prop:AnStrCX}), so it suffices to prove this result for $\cM=\cE^\varphi$ with $\varphi\in\Gamma(X;\cO_X(*D))$. (Recall that we have
	$(C_X(\cE^\varphi))^{\overline{\cA}}\iso (\cE^{-\overline{\varphi}})^{\overline{\cA}}$
	by Lemma~\ref{lemma:CXEphi}.)
	
	Then we get	
	\begin{align*}
	\DRE{\widetilde{\Xc}}\big(C_X(\cE^\varphi)^{\overline{\cA}}\big)&\iso \DRE{\widetilde{\Xc}}\big((\cE^{-\overline{\varphi}})^{\overline{\cA}}\big)\\ 
	&\iso \EE\varpi^!\DRE{\Xc}(\cE^{-\overline{\varphi}})\\
	&\iso \C^\EE_{\widetilde{\Xc}}\conv \RIhom(\pi^{-1}\C_{X\setminus D},\C_{\{t=-\real \overline{\varphi}(x)\}})[d_X]
	\end{align*}
	by \cite[Corollary~9.2.3 and Lemma~9.3.1]{DK16}.
	On the other hand, we similarly have
	\begin{align*}
		\SolE{\widetilde{X}}\big((\cE^\varphi)^\cA\big) &\iso \EE\varpi^! \RIhom(\pi^{-1}\C_{X\setminus D},\SolE{X}(\cE^\varphi))\\
		&\iso \C^\EE_{\widetilde{X}} \conv \RIhom(\pi^{-1}\C_{X\setminus D},\C_{\{t=-\real \varphi(x)\}}).
	\end{align*}
	Since conjugation does not affect the real part of a holomorphic function, we see that both sides of \eqref{eq:morphDRECSolEbu} are isomorphic in the case $\cM=\cE^\varphi$.
	
	Applying now $\pi^{-1}\C_{X\setminus D}\otimes \EE\varpi_*$ to \eqref{eq:morphDRECSolEbu}, we obtain (see \cite[Corollary~9.2.3]{DK16} and \cite[§3]{IT20})
	$$\DRE{\Xc}\big(C_X(\cM)(!D)\big)[-d_X]\overset{\sim}{\To} \SolE{X}(\cM(*D)),$$
	and noting that $\cM$ (having good normal form) is meromorphic at $D$, while $C_X(\cM) \iso C_X(\cM)(!D)$ by Lemma~\ref{lemma:ConjLoc}, we obtain the desired isomorphism.
\end{proof}

It remains to prove the following lemma, which has been used in the above proof.

\begin{lemma}\label{lemma:morphDRECSolEbu}
	Let $X$ be a complex manifold and $D\subset X$ a normal crossing divisor and $\widetilde{X}$, $\widetilde{\Xc}$ the associated real blow-up spaces along $D$. Let $\cM\in\DbcohD{X}$. Then there is a canonical morphism
	$$\DRE{\widetilde{\Xc}}\big(C_X(\cM)^{\overline{\cA}}\big)[-d_X] \to \SolE{\widetilde{X}}(\cM^\cA).$$
\end{lemma}
\begin{proof}
	Again, we use the notation in \cite{SabAst}, in particular the sheaf $$\mathcal{D}b^{\mathrm{mod}\, D}_{\widetilde{X}}\iso \varpi^{-1}\cO_X(*D)\otimes_{\varpi^{-1}\cO_X} \Db{\widetilde{X}}.$$
	
	First, note that there is a canonical morphism
	\begin{align*}
		C_X(\cM)^{\overline{\cA}} &\iso \RHom_{\varpi^{-1}\cD_X}(\varpi^{-1}\cM,\cA^{\mathrm{mod}}_{\widetilde{\Xc}}\otimes_{\varpi^{-1}\cO_{\Xc}} \varpi^{-1}\Db{X})\\
		&\to \RHom_{\cD_{\widetilde{X}}^\cA}(\cM^\cA,\mathcal{D}b^{\mathrm{mod}\, D}_{\widetilde{X}})
	\end{align*}
	for $\cM\in\DbcohD{X}$, induced by the natural morphism $\varpi^{-1}\Db{X}\to\mathcal{D}b^{\mathrm{mod}\, D}_{\widetilde{X}}$ and the fact that $\mathcal{D}b^{\mathrm{mod}\, D}_{\widetilde{X}}$ is a module over $\cA^{\mathrm{mod}}_{\widetilde{\Xc}}$.
	
	It therefore suffices to find a morphism
	$$\DRE{\widetilde{\Xc}}\big(\RHom_{\cD_{\widetilde{X}}^\cA}(\cM^\cA,\mathcal{D}b^{\mathrm{mod}\, D}_{\widetilde{X}_\R})\big)[-d_X] \to \SolE{\widetilde{X}}(\cM^\cA).$$
	
	Similarly to the proof of Lemma~\ref{lemma:morphDRECSolE}, it is enough to give a morphism
	$$\varpi^{-1}\Omega_{\Xc}\Ltens{\varpi^{-1}\cD_{\Xc}}(\OE{\widetilde{\Xc}}\Ltens{\cO_{\Xc}} \cD b^{\mathrm{mod}\,D}_{\widetilde{X}_\R})\To \OE{\widetilde{X}}.$$
	By a reduction similar to that in Lemma~\ref{lemma:morphDRECSolE}, such a morphism is eventually induced by a morphism
	$$\Cinft{X\times\PP}\Ltens{\beta\cO_{\Xc}} \Dbt{X}\To \Dbt{X\times\PP},$$
	which is given by multiplication of tempered smooth functions with tempered distributions.	
	We leave details to the reader.
\end{proof}

With this in hand, we can now proceed completely analogously to \cite{KasConj} to define the conjugation functor $c$, the functor corresponding to complex conjugation on constructible sheaves via the Riemann--Hilbert correspondence.

We define
\begin{align*}
c\colon \DbD{X} &\to \DbD{X}\\
\cM&\mapsto c(\cM) \vcentcolon= C_{\Xc}(\D_{\Xc}(\overline{\cM})),
\end{align*}
(recall the definition of $\overline{\cM}$ from Section~\ref{subsec:GalConj}).

\begin{thm}\label{prop:DREconj}
	Let $X$ be a complex manifold and $\cM\in\DbholD{X}$. Then there is an isomorphism in $\EbIC{X}$
	$$\DRE{X}(c(\cM)) \iso \overline{\DRE{X}(\cM)}.$$
	Similarly, there is an isomorphism $\SolE{X}(c(\cM)) \iso \overline{\SolE{X}(\cM)}$.
\end{thm}
\begin{proof}
	It is easy to check that
	\begin{align*}
		\DRE{X}(C_{\Xc}(\D_{\Xc}(\overline{\cM})))&\iso \SolE{\Xc}(\D_{\Xc}(\overline{\cM}))\\
		&\iso \DRE{\Xc}(\overline{\cM})\iso \overline{\DRE{X}(\cM)},
	\end{align*}
which follows from Proposition~\ref{prop:DRECSolEIso}.

The second claim follows by applying the duality functor (and a shift by $-d_X$) since complex conjugation commutes with duality (cf.\ \cite[Lemma~2.3]{BHHS22}).
\end{proof}

To end this section, let us come back to the proof of Kashiwara's conjecture \cite[Remark~3.5]{KasConj}.
The following theorem is the first and main part of this conjecture.
\begin{thm}[{see \cite[Theorem 3.1.2]{SabAst}, \cite[Corollary~4.19]{MocStokesMero}}]\label{thm:SabbahCXhol}
	Let $X$ be a complex manifold and $\cM\in \ModholD{X}$ a holonomic $\cD_X$-module. Then $$C_X(\cM)= \RR\Hom_{\cD_X}(\cM,\Db{X})$$ is concentrated in degree $0$ and holonomic.
\end{thm}

The other parts of the original conjecture of M.\ Kashiwara follow from this theorem (see \cite[§3.1]{SabAst}). They are stated in the (first part of the) following corollary, and with the help of our results from this section, we give a different way of deducing it from Theorem~\ref{thm:SabbahCXhol}.
\begin{cor}
	The functor $C_X\colon \DbholD{X}\to \DbholD{\Xc}$ is an equivalence of categories with inverse $C_{\Xc}$. The functor $c\colon \DbholD{X}\to\DbholD{X}$ is an equivalence of categories and is its own inverse.
\end{cor}
\begin{proof}
	Let $\cM\in\DbholD{X}$. The fact that $C_X(\cM)\in\DbholD{\Xc}$ and hence $c(\cM)\in\DbholD{X}$ follows easily from Theorem~\ref{thm:SabbahCXhol} by induction on the amplitude of a holonomic complex. Then we have
	\begin{align*}
		\DRE{X}(C_{\Xc}(C_X(\cM))) &\iso \SolE{\Xc}(C_X(\cM)) \iso \DD_{\Xc}\DRE{\Xc}(C_X(\cM))\\
		&\iso \DD_{X}\SolE{X}(\cM)\iso \DRE{X}(\cM)
	\end{align*}
	by Proposition~\ref{prop:DRECSolEIso} and
	\begin{align*}
		\DRE{X}\big(c(c(\cM))\big)\iso \overline{\DRE{X}(c(\cM))}\iso \DRE{X}(\cM)
	\end{align*}
	by Theorem~\ref{prop:DREconj}, and by the Riemann--Hilbert correspondence of \cite{DK16}, we get the isomorphisms $C_{\Xc}(C_X(\cM))\iso \cM$ and $c(c(\cM))\iso\cM$.
\end{proof}

\section{Galois descent for enhanced ind-sheaves}\label{sec:GaloisDescentEb}

In this section, we give some complements on the study of Galois descent for enhanced ind-sheaves done in \cite{BHHS22}. We mostly reformulate what has been established there, slightly generalizing the descent statement by removing the compactness assumption that was present in \cite[Proposition 2.15]{BHHS22}. We mainly study $\R$-constructible enhanced ind-sheaves here.

The starting point is the functor of extension of scalars: Let $\cX$ be a bordered space and let $L/K$ be a field extension, then we have the functor
$$\EbIK{\cX}\to \EbIL{\cX}, \quad H\mapsto \pi^{-1}L_X \otimes_{\pi^{-1}K_X} H.$$
Its compatibilities with direct and inverse images have been described in \cite{BHHS22}. We restate them here, removing some restrictions that are not necessary.
\begin{lemma}\label{lemma:compIm}
	Let $\cX$ and $\cY$ be bordered spaces and let $f\colon \cX\to\cY$ be a morphism. Let $F,F_1,F_2\in\EbIK{\cX}$ and $G\in\EbIK{\cY}$.
	Then we have isomorphisms
	\begin{align*}
	\piLK{\cY}\EE f_{!!} F &\iso \EE f_{!!} (\piLK{\cX}F)\\
	\piLK{\cX}\EE f^{-1} G &\iso \EE f^{-1}(\piLK{\cY}G).
	\end{align*}
If $\cX$ and $\cY$ are real analytic bordered spaces, $f$ is a morphism of real analytic bordered spaces and $F\in\EbRcIK{\cX}$, $G\in\EbRcIL{\cY}$, we also have isomorphisms
	\begin{align*}
		\piLK{\cX} \DE{\cX} F&\iso \DE{\cX} (\piLK{\cX}F),\\
		\piLK{\cX}\EE f^{!} G&\iso \EE f^{!}(\piLK{\cY}G).
	\end{align*}
If moreover $f$ is semi-proper, we have an isomorphism
	\begin{align*}
	\piLK{\cY}\EE f_{*} F&\iso \EE f_{*} (\piLK{\cX}F).
\end{align*}
\end{lemma}

\begin{rem}
	The condition for the isomorphism for the direct image $\EE f_*$ is in particular satisfied if $X\subseteq \widehat{X}$ and $Y\subseteq\widehat{Y}$ are relatively compact. 
	This means that the bordered spaces $\cX$ and $\cY$ are of a particular type, namely they are b-analytic manifolds and $F$ and $G$ are b-constructible, notions that have been investigated in \cite{Sch23}.
	
	Such functorialities are, of course, not restricted to the tensor product with a field extension, but this is the case we are interested in here. One could also try to perform a study of functorialities for more general sheaves $L_X$, similar to \cite{HS23}.
\end{rem}

In order to understand homomorphisms between scalar extensions of enhanced ind-sheaves, we prove the following two statements.

\begin{prop}\label{prop:HomEext}
	Let $L/K$ be a field extension. Let $\cX$ be a real analytic bordered space 
	and let $F,G\in\EbRcIK{\cX}$. Then there is an isomorphism in $\DbL{X}$
	$$\RHomE_{L_\cX}(\piLK{\cX}F,\piLK{X}G)\iso L_X\otimes_{K_X}\RHomE_{K_\cX}(F,G).$$
\end{prop}
\begin{proof}
	There exists a canonical morphism (from right to left), and it suffices to prove that it is an isomorphism locally on an open cover of $X$ by relatively compact subanalytic subsets. In other words, it suffices to prove the isomorphism under the assumption that $F\iso K^\EE_\cX\conv \cF$, $G\iso K^\EE_\cX\conv \cG$ for some $\cF,\cG\in\DbRcK{\cX\times\R_\infty}$. Let us, without loss of generality, choose them such that $K_{\{t\geq 0\}}\conv \cF\iso \cF$ (and similarly for $\cG$).
	Then
	\begin{align*}
		\RHomE_{L_\cX}&(\piLK{\cX}F,\piLK{\cX}G)\\ &\iso
		\RHomE_{L_\cX}\big(\piLK{\cX}\cF,L^\EE_\cX\conv(\piLK{\cX}\cG)\big)\\
		&\iso \alpha_{\cX}\RR {\pi_{\cX}}_*\RIhom_{L_\cX}\big(\piLK{\cX}\cF,\piLK{\cX}(K^\EE_\cX\conv\cG)\big)\\
		&\iso L_{\cX}\otimes_{K_{\cX}} \alpha_{\cX}\RR {\pi_{\cX}}_*\RIhom_{K_\cX}(\cF,K^\EE_\cX\conv\cG)\\
		&\iso L_{\cX}\otimes_{K_{\cX}}\RHomE_{K_\cX}(F,G).
	\end{align*}
	Here, the first isomorphism follows from \cite[Lemma~4.5.15]{DK16}, and the second isomorphism uses the definition of $\RHomE$. The third isomorphism follows from \cite[Lemma 2.7]{BHHS22} and \cite[Lemmas 3.3.12 and 3.3.7]{DK16} (note that the extended map $\overline{\pi}\colon \widehat{X}\times\overline{\R}\to\widehat{X}$ is proper).
\end{proof}

\begin{rem}
	In the above proof, we have used \cite[Lemma~2.7]{BHHS22}, i.e.\ the compatibility of $\RIhom$ with extension of scalars. The proof of this lemma in loc.~cit.\ relied on a technical lemma whose details were not expanded there. Let us note that an alternative proof of this lemma can be performed by using the results of \cite{HS23}.
\end{rem}

\begin{cor}\label{cor:morphExt}
	Let $L/K$ be a field extension. Let $\cX$ be a real analytic bordered space and let $F,G\in\EbRcIK{\cX}$.
	If $L/K$ is finite or $X\subset \widehat{X}$ is relatively compact, then the natural morphism
	$$L\otimes_K \sHom_{\EbRcIK{\cX}}(F, G)\to\sHom_{\EbRcIL{\cX}}(\piLK{\cX} F, \piLK{\cX} G)$$
	is an isomorphism.
\end{cor}
\begin{proof}
	It follows from \cite[(2.6.3)]{DK19} (cf.\ also \cite[Proposition~3.2.9 and Corollary~3.2.10]{DK16}) that
	\begin{align*}
	\sHom_{\EbRcIK{\cX}}(F, G) &\iso \sHom_{\EbRcIK{\widehat{X}}}(\EE {j_{\cX}}_{!!}F, \EE {j_{\cX}}_{!!}G)\\
	&\iso H^0\RR\Gamma\big(\widehat{X};\RHomE_{K_{\widehat{X}}}(\EE {j_{\cX}}_{!!}F, \EE {j_{\cX}}_{!!}G)\big),
	\end{align*}
	where $j_{\cX}\colon \cX\to\widehat{X}$ is the natural morphism of bordered spaces,	and analogously over $L$.
	By Lemma~\ref{lemma:compIm} and Proposition~\ref{prop:HomEext}, we know that extension of scalars commutes with $\EE {j_{\cX}}_{!!}$ and $\RHomE$. It therefore remains to show that it commutes with taking global sections. In the case of a finite field extension, taking global sections commutes with extension of scalars (since finite direct sums are finite direct products and direct images are right adjoints). In the case where $X\subset \widehat{X}$ is relatively compact, the sheaf we take sections of has compact support, so global sections are a proper direct image and hence commute with extension of scalars.
\end{proof}

\begin{lemma}\label{lemma:extConservative}
	Let $\cX$ be a real analytic bordered space and $L/K$ a field extension. Let $H\in\EbRcIK{\cX}$ such that $\piLK{\cX}H\iso 0$. Then $H\iso 0$. In particular, extension of scalars is a conservative functor on $\EbRcIK{\cX}$.
\end{lemma}
\begin{proof}
	In the situation of Corollary~\ref{cor:morphExt}, if $X\subset \widehat{X}$ is relatively compact, then for a morphism $f\colon F\to G$, we have that $1\otimes f$ is an isomorphism if and only if $f$ is so: Let $1\otimes g_0+\sum_{i=1}^n \ell_i\otimes g_i$ be an inverse, where $1,\ell_1,\ldots,\ell_n$ are elements of $L$, linearly independent over $K$, and $g_i\colon G\to F$ are morphisms. Then $$1\otimes\id_F=\big(1\otimes g_0+\sum_{i=1}^{n}\ell_i\otimes g_i\big)\circ(1\otimes f) = 1\otimes (g_0\circ f)+\sum_{i=1}^{n} \ell_i\otimes (g_i\circ f)$$ and hence $g_0\circ f=\id_F$, and analogously $f\circ g_0=\id_G$.
	
	Now let $f\colon F\to 0$ be the unique morphism. Then we can cover $X$ by subsets $U$ that are relatively compact in $\widehat{X}$ and on each $U_\infty=(U,\overline{U})$ we get that $F|_{U_\infty}\iso 0$. (Here, the closure is taken in $\widehat{X}$.) By \cite[Proposition~3.8]{DKsh}, this suffices to conclude $F\iso 0$ globally.
\end{proof}

Let us also state the following lemma, whose proof is completely analogous to \cite[Proposition 4.9]{Ho23GalShv}.
\begin{lemma}\label{lemma:morphGalExt}
	Let $L/K$ be a finite Galois extension with Galois group $G$. Let $\cX$ be a real analytic bordered space and let $F,G\in\EbRcIK{\cX}$. Then the subspace of
	$$\sHom_{\EbRcIL{\cX}}(\piLK{\cX} F, \piLK{\cX} G)\iso L\otimes_K \sHom_{\EbRcIK{\cX}}(F, G)$$
	consisting of morphisms $f$ that fit, for every $g\in G$, into a commutative diagram
	$$\begin{tikzcd}
		\piLK{\cX} F \arrow{r}{f}\arrow{d}{g\otimes\id_F}& \piLK{\cX} G\arrow{d}{g\otimes\id_G}\\
		\pi^{-1}\overline{L}^g_\cX \otimes_{\pi^{-1}K_\cX} F\arrow{r}{\overline{f}^g} &\pi^{-1}\overline{L}^g_\cX \otimes_{\pi^{-1}K_\cX} G
	\end{tikzcd}$$
	is exactly the $K$-vector space $1\otimes \sHom_{\EbRcIK{X}}(F, G)$.
\end{lemma}

We can now give a slightly more general version (without the compactness assumption) of \cite[Proposition 2.15]{BHHS22}. Recall that if $L/K$ is a finite Galois extension with Galois group $G$ and $H\in \EbIL{\cX}$, then for any $g\in G$ one has the $g$-conjugate $\overline{H}^g\in\EbIL{\cX}$. A $G$-structure on $H$ is a collection $(\varphi_g)_{g\in G}$ of isomorphisms $\varphi_g\colon H\overset{\sim}{\to} \overline{H}^g$ such that for any $g,h\in G$ one has $\overline{\varphi_g}^h\circ\varphi_h=\varphi_{gh}$ (see \cite[Definition~2.12]{BHHS22} for this notion). Note that for $H_K\in\EbIK{\cX}$, the action of $G=\mathrm{Aut}(L/K)$ on $L$ induces a natural $G$-structure on $\piLK{\cX}H_K$.
\begin{prop}\label{prop:descRc0}
	Let $L/K$ be a finite Galois extension with Galois group $G$. Let $\cX$ be a real analytic bordered space and let $H\in\EzRcIK{\cX}$ be equipped with a $G$-structure. Then there exists $H_K\in\EzRcIK{\cX}$ and an isomorphism $$H\iso \piLK{\cX}H_K$$ through which the given $G$-structure on $H$ coincides with the natural one on\linebreak $\piLK{\cX}H_K$.
\end{prop}
\begin{proof}
	In \cite[Proposition 2.15]{BHHS22}, we have proved the result in the case that $\cX=(X,\widehat{X})$ is such that $X$ is relatively compact in $\widehat{X}$. (The part of the statement about the $G$-structures has not explicitly been mentioned in loc.~cit.\ but is clear from the construction.)
	
	Now, if $X\subseteq \widehat{X}$ is not relatively compact, we can cover $X$ by open subsets $U_i$, $i\in I$, all of which are relatively compact in $\widehat{X}$. On each $(U_i)_\infty=(U_i,\overline{U_i})$, we can perform the above construction (closures are taken in $\widehat{X}$ here) and obtain $H_{K,i}\in\EzRcIK{(U_i)_\infty}$ with $H|_{(U_i)_\infty}\iso \piLK{(U_i)_\infty} H_{K,i}$. Further, on any overlap $(U_{ij})_\infty\vcentcolon=(U_i\cap U_j,\overline{U_i\cap U_j})$ we have an isomorphism
	$$\piLK{(U_{ij})_\infty} H_{K,i}|_{(U_{ij})_\infty}\iso\piLK{(U_{ij})_\infty} H_{K,j}|_{(U_{ij})_\infty}$$
	induced by the identity on $H$ (which is certainly compatible with the given $G$-structure). Since the natural $G$-structures on both sides correspond to the given one on $H$, we see that this morphism is compatible with the natural $G$-structures on both sides. Hence, by Lemma~\ref{lemma:morphGalExt}, it descends to an isomorphism $H_{K,i}|_{(U_{ij})_\infty}\iso H_{K,j}|_{(U_{ij})_\infty}$. As $U\mapsto \EzRcIK{U_\infty}$ is a stack, this yields an object $H_K\in \EzRcIK{\cX}$ as desired.
\end{proof}

With these results in hand, we can now state Galois descent for $\R$-constructible enhanced ind-sheaves concentrated in one degree as an equivalence of categories.

We denote by $\EzRcIL{\cX}^G$ the category of pairs $(H,(\varphi_g)_{g\in G})$ of objects $H\in\EzRcIL{\cX}$ together with a $G$-structure. A morphism $(H,(\varphi_g)_{g\in G})\to (H',(\varphi'_g)_{g\in G})$ is a morphism $f\colon H\to H'$ such that for any $g\in G$ the following diagram commutes:
$$\begin{tikzcd}
	H\arrow{r}{f}\arrow{d}{\varphi_g} & H'\arrow{d}{\varphi'_g}\\
	\overline{H}^g\arrow{r}{\overline{f}^g} &\overline{H'}^g
\end{tikzcd}$$
\begin{thm}\label{thm:GaloisDescentEb}
	Let $L/K$ be a finite Galois extension. Let $\cX$ be a real analytic bordered space. Then extension of scalars induces an equivalence of categories
	$$\EzRcIK{\cX}\overset{\sim}{\To}\EzRcIL{\cX}^G.$$
\end{thm}
\begin{proof}
	Full faithfulness follows from Lemma~\ref{lemma:morphGalExt} and essential surjectivity follows from Proposition~\ref{prop:descRc0}.
\end{proof}

\section{$K$-structures and monodromy data}\label{sec:Monodromy}

In this section, we are going to study the impact of a $K$-structure on the generalized monodromy data (in particular Stokes matrices) associated to an enhanced ind-sheaf or a meromorphic connection on a complex curve. The main descent result is Theorem~\ref{thm:HLTdescent}, and we draw some explicit consequences in the following subsections. Many other parts, in particular the description of generalized monodromy data and meromorphic connections, are mostly a reproduction of well-known facts in order to clarify our setting and language.

\begin{defi}
	Let $\cX$ be a bordered space. If $K\subset L$ is a subfield and $H\in\EbIL{\cX}$, then a $K$-lattice of $H$ is an enhanced ind-sheaf $H_K\in\EbIK{\cX}$ such that $H\iso \pi^{-1}L_X\otimes_{K_X} H_K$. If such an $H_K$ exists, we say that $H$ has a $K$-structure.
\end{defi}

\begin{ex}
	It is obvious that objects of the form $\E^f_{W|X}$ have a $K$-structure for any subfield $K\subset \C$ (recall these exponential objects from Section~\ref{sec:enhanced}). Similarly, let us consider the following situation: Let us consider smooth algebraic varieties $X$ and $S$ and morphisms
	$$\begin{tikzcd}
		X\arrow{d}{\varphi} \arrow{r}[swap]{g}& S\\
		\A^1
	\end{tikzcd}$$
	and the (algebraic) $\cD_X$-module $\cE^\varphi$ on $X$ as well as its direct image $g_+\cE^f$ on $S$. We can consider these modules (still denoting them by the same symbols) as analytic D-modules on a smooth completion $\widehat{X}$ of $X$ and $\widehat{S}$ of $S$, respectively, by taking the analytifications of the meromorphic extensions, and we can also extend $g$ as a map $\hat{g}\colon \widehat{X}\to\widehat{S}$ between these completions. Then, the enhanced ind-sheaf $\SolE{\widehat{S}}(g_+\cE^f)\iso \EE \hat{g}_{!!}\SolE{\widehat{X}}(\cE^f)$ has a $K$-structure over any field $K\subset \C$.

	Using Galois descent, we have also deduced certain $K$-structures of the enhanced solutions of hypergeometric systems in \cite{BHHS22}.
\end{ex}

In the rest of this section, $X$ will be a Riemann surface (a connected one-di\-men\-sion\-al complex manifold) and $D\subset X$ a discrete set of points.
We write $X^*\vcentcolon= X\setminus D$.
Given charts around every point $p\in D$ and $\beps = (\varepsilon_p)_{p\in D}$, where $0\leq\varepsilon_p\ll 1$ (we just write $0\leq\beps\ll 1$, and similarly $0<\beps\ll 1$), we set $X^*_{\beps}\vcentcolon= X\setminus\bigcup_{p\in D} \overline{B}_{\varepsilon_p}(p)$, where $\overline{B}_{\varepsilon_p}(p)$ is the closed ball around $p$ of radius $\varepsilon_p$ in the given chart around $p$. (Here, one should read ``$\beps\ll 1$'' as ``the $\varepsilon_p$ are sufficiently small'', in the sense that $\overline{B}_{\varepsilon_p}(p)$ does not have a limit point in the boundary of the chart around $p$ and that $\overline{B}_{\varepsilon_p}(p)\cap \overline{B}_{\varepsilon_q}(q)=\emptyset$ for $p,q\in D$ with $p\neq q$. This ensures that $X^*_{\beps}$ is an open subset of $X$ and that it is homotopy equivalent to $X^*$.)

Let $\k$ be a field. Recall the notation for enhanced ind-sheaves from Section~\ref{sec:enhanced}, and in particular the natural embedding $\e{\k}{X}$ from sheaves to ind-sheaves as well as the notation for enhanced exponentials. We will simply write $\E^{\real\varphi}_{W,\k}$ or even $\E^{\real\varphi}_{W}$ (if the field $\k$ is clear from the context) instead of $\E^{\real\varphi}_{W|X,\k}$ to simplify notation.

By a \emph{sector} with vertex at $p\in X$, we will mean a subset of $X$ of the following form: Given a local coordinate $z$ in a neighbourhood $U$ of $p$ with $z(p)=0$, an angle $\theta\in\R/2\Z$ and $r,\delta\in\R_{>0}$, we set
$$S_\theta(r,\delta)\vcentcolon= \{z\in U\mid 0< |z|<r, \theta-\delta<\arg z<\theta+\delta\}.$$

\subsection{Descent of meromorphic normal form}

Let $X$ be a Riemann surface. We first investigate objects motivated by the topological counterpart of meromorphic connections on $X$ via the Riemann--Hilbert correspondence of \cite{DK16}. This type of objects is described in the following definition, which is strongly motivated by \cite{MocCurveTest2} and \cite[§6.2]{DKmicrolocal}. The goal of this subsection is to show that any lattice inherits this property.

\begin{defi}\label{def:HLTtype}
	Let $H\in \EbIk{X}$. We say that $H$ is of \emph{meromorphic normal form} with respect to $D$ if the following conditions are satisfied:
	\begin{itemize}
		\item[(a)] $\pi^{-1}\k_{X^*}\otimes H \iso H$.
		\item[(b)] There exists a local system $\cL$ on $X^*$ (consider it extended by zero to $X$, i.e.\ $\cL\in\Mod{\k_X}$) such that for any $0<\beps\ll 1$, we have $$\pi^{-1}\k_{X^*_{\beps}}\otimes H \iso \e{\k}{X}(\cL_{X^*_{\beps}}).$$
		\item[(c)] For any $p\in D$, any local coordinate $z$ in a neighbourhood $U$ of $p$ and any direction $\theta\in\R/2\pi\Z$, there exist $r,\delta\in \R_{>0}$, $n\in \Z_{>0}$, a determination of $z^{1/n}$ on the sector $S_\theta\vcentcolon=S_\theta(r,\delta)$, a finite set $\Phi_\theta\subset z^{-1/n}\C[z^{-1/n}]$ and an integer $r_\varphi\in \Z_{>0}$ for any $\varphi \in \Phi_\theta$ such that (note that the $\varphi\in \Phi_\theta$ define holomorphic functions on $S_\theta$) one has an isomorphism
		$$\pi^{-1}\k_{S_\theta}\otimes H \iso \bigoplus_{\varphi\in \Phi_\theta}(\E^{\real \varphi}_{S_\theta})^{r_\varphi}.$$
	\end{itemize}
\end{defi}

In simple terms, an object of meromorphic normal form is one that is localized at $D$, looks like a local system away from the singularities and satisfies a ``Stokes phenomenon'' close to the singularity, i.e.\ it decomposes on small sectors as a direct sum of exponentials.

It is not difficult to see that if $H\in \EbIk{X}$ satisfies the conditions (a), (b) and (c), then it already follows that it is concentrated in degree $0$ and $\R$-constructible (cf.\ \cite[Lemma~4.9.9]{DK16}), i.e.\ $H\in\EzRcIk{X}$.

\begin{rem}\label{rem:finSectors}
	Let us make some remarks that will be useful later. First of all, this definition is really just a global reformulation of \cite[§6.2]{DKmicrolocal}. The objects defined above can also be described as those enhanced perverse ind-sheaves (see \cite[Definition~3.4]{DK23}) satisfying $\pi^{-1}\k_{X^*}\otimes H\iso H$, and shifted by $-1$. On the other hand, the notion of meromorphic normal form also corresponds to that of ``quasi-normal form'' (in dimension $1$ and for $\k=\C$) used in \cite{Ito}. If $\k=\C$, objects of meromorphic normal form are the essential image of meromorphic connections on $X$ with poles at $D$ under the enhanced solution functor $\SolE{X}$ (cf.\ Proposition~\ref{prop:MeroConnNormalForm} below).
	
	To make it more similar to those definitions, our condition (b) of Definition~\ref{def:HLTtype} can be reformulated as follows:
	\begin{itemize}
		\item[(b')] There is an isomorphism $H|_{X^*}\iso \k^\EE_{X^*}\otimes\pi^{-1}\cL$ for a local system $\cL$ on $X^*$.
	\end{itemize}
	This is equivalent to (b) since $\EzIk{X^*}$ is a stack and hence the isomorphisms on all the $X^*_{\beps}$ imply the isomorphism on the whole of $X^*$. (We opted for the formulation (b) in our definition in order to express (b) and (c) in a more similar manner.)
	
	Since the circle of directions $\R/2\pi\Z$ is compact, part (c) of Definition~\ref{def:HLTtype} can equivalently be formulated as follows:
	\begin{itemize}
		\item[(c')] \textit{For any $p\in D$ and a local coordinate $z$ in a neighbourhood $U$ of $p$, there exist a positive integer $n\in\Z_{>0}$ and a finite number of sectors $S^p_1,\ldots,S^p_{k_p}$ (for some $k_p\in\Z_{>0}$) covering a punctured neighbourhood of $p$ such that for any $j\in\{1,\ldots,k_p\}$ one has: a determination of $z^{1/n}$ on the sector $S^p_j$, a finite set $\Phi^p_j\subset z^{-1/n}\C[z^{-1/n}]$ and an integer $r_{\varphi}\in \Z_{>0}$ for any $\varphi\in \Phi^p_j$ such that (note that the $\varphi\in \Phi^p_j$ define holomorphic functions on $S^p_j$) one has an isomorphism}
		\begin{equation}\label{eq:HLTisoOnSectors}
		\pi^{-1}\k_{S^p_j}\otimes H \iso \bigoplus_{\varphi\in \Phi^p_j}(\E^{\real \varphi}_{S^p_j})^{r_{\varphi}}.
		\end{equation}
	\end{itemize}

	We will assume that the sectors $S^p_1,\ldots,S^p_{k_p}$ are ordered in a counter-clockwise sense around $p$ (with respect to their central directions). Then, for any $j\in\{1,\ldots,k_p\}$ we can consider the overlap $S^p_{j,j+1}\vcentcolon=S^p_j\cap S^p_{j+1}$ (where, of course, we count modulo $k_p$, so that $k_p+1\vcentcolon= 1$). On this overlap we have the two isomorphisms induced by the ones from \eqref{eq:HLTisoOnSectors}
	$$ \bigoplus_{\varphi\in \Phi^p_j}(\E^{\real \varphi}_{S^p_{j,j+1}})^{r_\varphi} \iso \pi^{-1}\k_{S^p_{j,j+1}}\otimes H \iso \bigoplus_{\psi\in \Phi^p_{j+1}}(\E^{\real \psi}_{S^p_{j,j+1}})^{r_\psi}.$$
	We can therefore identify the index sets $\Phi^p_j$ and $\Phi^p_{j+1}$ (see \cite[Lemma~3.25]{MocCurveTest2}, cf.\ also \cite[Corollary~5.2.3]{DKmicrolocal}), and also $r_\varphi=r_\psi$ via this identification. More precisely, this identification is given by holomorphic continuation of functions in counter-clockwise direction from $S^p_j$ to $S^p_{j+1}$. We will hence consider all the direct sums in the isomorphisms \eqref{eq:HLTisoOnSectors} to be indexed by the same set $\Phi^p=\Phi^p_1\subset z^{-1/n}\C[z^{-1/n}]$.
	In particular, this also means that $\Phi^p$ is closed under analytic continuation around the circle, i.e.\ if $\varphi(z^{-1/n})\in\Phi^p$, then $\varphi(e^{\frac{2\pi i}{n}}z^{-1/n})\in \Phi^p$.
	
	Denoting by $r$ the rank of the local system $\cL$ from (b), it is also clear that $r=\sum_{\varphi\in\Phi^p} r_\varphi$ for any $p$.
\end{rem}

Our goal is to show that the property of being of meromorphic normal form descends to a lattice. To do this, let us prepare some lemmas. Essentially, the two lemmas below are concerned with the descent of direct sums of exponentials (Lemma~\ref{lemma:latticeE}) and descent of enhanced local systems (Lemma~\ref{lemma:localSystemDescent}) separately, and we will combine them to obtain our desired result (Theorem~\ref{thm:HLTdescent}).

\begin{rem}\label{rem:f<g}
	To prepare for the argument that follows, let us make the following easy observation:
	
	Let $X$ be an open disc around a point $p\in\C$, let $S\subset X$ be an open sector with vertex $p$, and let us consider holomorphic functions $\varphi, \psi\colon S\to \C$ given by elements $\varphi,\psi\in z^{-1/n}\C[z^{-1/n}]$ for some choice $z^{-1/n}$ of an $n$-th root of a local coordinate $z$ at $p$. We assume that $\psi\prec_S \varphi$, i.e.\ $\real(\varphi-\psi)$ is bounded from below near $p$, but unbounded from above near $p$. Then:
	\begin{itemize}
		\item For any relatively compact subanalytic open subset $U\subseteq X\times\overline{\R}$ such that $\E^{\real\psi}_S(U)\neq 0$, we have $\E^{\real\varphi}_S(U)\neq 0$.
	\end{itemize}
	To see this, note that $\E^{\real \psi}_S(U)\neq 0$ means in particular that $U$ intersects any $\{t\geq -\real\psi + a\}$ nontrivially. Now let us consider an arbitrary $b\in \R$, then there exists $a\in\R$ such that $\real(\varphi-\psi) > b-a$, so $\{t\geq -\real \varphi + b\}\supseteq \{t\geq -\real\psi + a\}$. This means that also $\E^{\real\varphi}_S(U)\neq 0$.
	\begin{itemize}
		\item There exists a relatively compact subanalytic open subset $U\subseteq X\times\overline{\R}$ such that $\E^{\real\varphi}_S(U)=\k$ but $\E^{\real\psi}_S(U)=0$.
	\end{itemize}
	Such a set can be given explicitly: Let $V\subseteq X$ an open relatively compact set such that $\real(\varphi-\psi)$ is unbounded from above on $V$. Set $U\vcentcolon= \{(x,t)\in X\times \overline{\R}\mid x\in V, t\in\R, t<-\real\psi\}$.
	Certainly, $U$ does not intersect $\{t\geq -\real\psi+a\}$ for any $a>0$, so $\E^{\real\psi}_S(U)=0$. On the other hand, for any $a\in \R$ there exists $x\in V$ such that $\real\varphi(x)-\real\psi(x)>a$, and hence $U$ intersects any $\{t\geq -\real\varphi+a\}$. Note that $V$ can be chosen such that this intersection has only one ``relevant'' connected component (i.e.\ one connected component with $p$ in its boundary, which ``survives'' for $a\to\infty$), and so $\E^{\real\varphi}_S(U)=\k$.
\end{rem}

\begin{lemma}\label{lemma:latticeE}
	Let $L/K$ be a field extension and $X$ a Riemann surface. Let $S\subset X$ be an open sector with vertex at some $p\in X$ and let $\Phi\subset z^{-1/n}\C[z^{-1/n}]$ be a finite set of holomorphic functions on $S$ for some $n$-th root $z^{-1/n}$ of a local coordinate $z$ at $p$.
	Then, up to an automorphism of $\bigoplus_{\varphi\in \Phi} (\E^{\real\varphi}_{S,L})^{r_\varphi}$, any $K$-lattice $F_K\subset \bigoplus_{\varphi\in \Phi} (\E^{\real\varphi}_{S,L})^{r_\varphi}$ with $F_K\in\EbRcIK{X}$ is of the form $F_K= \bigoplus_{\varphi\in \Phi} (\E^{\real\varphi}_{S,K})^{r_\varphi}$.
\end{lemma}
\begin{proof}
	It is clear from Lemma~\ref{lemma:extConservative} that $F_K$ is concentrated in degree $0$.
	Let us assume for simplicity that $r_\varphi=1$ for any $\varphi\in\Phi$, and let us think of $\E^{\real\varphi}_{S,L}$ and $F_K$ as subanalytic sheaves on $X\times\overline{\R}$ (cf.\ Remark~\ref{rem:EfSuban}). Then for any open $U\subseteq X\times\overline{\R}$, the sections of $F_K$ on $U$ must be a $K$-lattice of the sections of  $\bigoplus_{\varphi\in \Phi} \E^{\real\varphi}_{S,L}$ on $U$ (see Lemma~\ref{lemma:latticeSectionsSuban}).
	
	Fix a numbering on the elements of $\Phi=\{\varphi_1,\ldots,\varphi_n\}$. For any relatively compact subanalytic open $V\subseteq S$, the space of sections of $\bigoplus_{\varphi\in \Phi} \E^{\real\varphi}_{S,L}$ on $V\times \R$ is $L^n$.
	A $K$-lattice for this space of sections is given by a vector space
	$$\sum_{j=1}^n K\cdot v_j$$
	for some $v_j\in L^n$ (linearly independent over $L$). Since $S$ is connected and restriction maps of $\bigoplus_{\varphi\in \Phi} \E^{\real\varphi}_{S,L}$ (recall that they mainly consist of projections and inclusions) need to be compatible with those of $F_K$, it is not difficult to see that one can choose the same vectors $v_j$ for each such $V$, and this lattice will also determine the sections of $F_K$ on all other open subsets of $X\times\overline{\R}$.
	
	We will show that a suitable choice of the vectors $v_1,\ldots,v_n$ yields a matrix $(v_1,\ldots,v_n)$ that defines an automorphism of $\bigoplus_{\varphi\in \Phi} \E^{\real\varphi}_{S,L}$ (cf.\ Lemma~\ref{lemma:homSumE}). Under (the inverse of) this automorphism, the lattice $F_K$ (with basis given by $v_1,\ldots,v_n$) will then correspond to the lattice $\bigoplus_{\varphi\in \Phi} \E^{\real\varphi}_{S,L}$ (whose basis is given by the standard basis vectors $e_1,\ldots,e_n$).
	
	Assume that there exists an open subset $S'\subseteq S$ such that $\varphi_\ell\prec_{S'} \varphi_k$ for some $k,\ell\in\{1,\ldots,n\}$. By possibly shrinking $S'$, we can assume that we have a total ordering $\varphi_{j_1}\prec_{S'} \ldots \prec_{S'} \varphi_{j_n}$, where $\{j_1,\ldots,j_n\}=\{1,\ldots,n\}$. Now we see (as in Remark~\ref{rem:f<g}) that there are relatively compact subanalytic open sets $U_1,\ldots,U_n\subset X\times\overline{\R}$ with $U_j\subset U_{j'}$ whenever $\varphi_{j'}\prec_{S'}\varphi_j$, and such that $\E^{\real\varphi_{j'}}_{S,L}(U_{j})=0$ if $\varphi_{j'}\prec_{S'} \varphi_j$ and $\E^{\real\varphi_{j'}}_{S,L}(U_{j})=L$ if $\varphi_{j}\preceq_{S'} \varphi_{j'}$, and restriction maps between these $U_{j}$ are given by projections. This means that for any $i\in\{1,\ldots,n\}$, the vectors $\mathrm{pr}_{\geq j_i}(v_1), \ldots, \mathrm{pr}_{\geq j_i}(v_n)$ span a $K$-vector space of dimension $i$, since they form a $K$-lattice of sections over $U_i$. (Here, $\mathrm{pr}_{\geq j_i}$ is the projection to the components $j_i,\ldots,j_n\in\{1,\ldots,n\}$.) 
	
	We can therefore assume without loss of generality that, for any $i\in\{2,\ldots,n\}$, the $j_i$-th component of $v_1, \ldots, v_{i-1}$ vanishes, as shown by the following argument: First of all, the $j_n$-th entries of $v_1,\ldots,v_n$ span a $K$-vector space of dimension $1$, meaning that at least one of them is nonzero and all the nonzero ones are $K$-multiples of each other. We can hence renumber the $v_j$ such that the $j_n$-th entry of $v_n$ is nonzero, and we can subtract suitable $K$-multiples of $v_n$ from $v_1,\ldots, v_{n-1}$ in order to obtain vectors $v_1',\ldots,v_n'$ which define the same $K$-lattice as $v_1,\ldots,v_n$ but satisfy the following properties: The $j_n$-th entry of $v_n'$ is nonzero and the $j_n$-th entries of $v'_1,\ldots,v'_{n-1}$ vanish. 
	Next, consider the two-dimensional vectors consisting of the $j_{n-1}$-th and $j_n$-th entries of $v_1',\ldots,v_n'$. Since they form a $K$-vector space of dimension $2$ and the $j_n$-th entries of $v_1',\ldots,v_{n-1}'$ vanish by construction, the $j_{n-1}$-th entry of at least one of the vectors $v_1',\ldots,v_{n-1}'$ must be nonzero, and all the others are $K$-multiples of this nonzero entry. Proceeding in a similar way as above, we obtain vectors $v_1'',\ldots,v_n''$ that generate the same $K$-lattice as $v_1',\ldots,v_n'$ and satisfies the following properties: The $j_n$-th entry of $v_n''$ is nonzero (in fact, $v_n''=v_n'$) and the $j_{n-1}$-th entry of $v_{n-1}''$ is nonzero, while the $j_n$-th entries of $v_1'',\ldots,v_{n-1}''$ and the $j_{n-1}$-th entries of $v_1'',\ldots,v_{n-2}''$ vanish. Continuing this process shows that we can indeed assume the $j_i$-th entry of $v_1,\ldots,v_{i-1}$ to be zero, for every $i\in\{2,\ldots,n\}$. In other words, the matrix $(v_1,\ldots,v_n)$ is upper-triangular with invertible diagonal entries when we consider its rows ordered by the indices $j_1,\ldots,j_n$. In particular, the $k$-th entry of $v_\ell$ is zero (recall that we started by assuming $\varphi_\ell\prec_{S'} \varphi_k$).
	
	We can then perform a similar argument (modifying the lattice by taking suitable $K$-linear combinations of the $v_j$) for any other such $S'$ and $\varphi_\ell\prec_{S'} \varphi_k$, whose details we leave to the reader (in particular, one needs to check that the vanishing entries already constructed are preserved under this procedure, which is due to the triangular structure of the matrix with respect to all the orderings already considered). Recalling Lemma~\ref{lemma:homSumE}, we finally see that the matrix $(v_1,\ldots,v_n)$ defines an automorphism of $\bigoplus_{\varphi\in \Phi} \E^{\real\varphi}_{S,L}$, and we can apply its inverse in order to transform the $K$-structure $F_K$ into $\bigoplus_{\varphi\in \Phi} \E^{\real\varphi}_{S,K}$, as desired.
	
	The proof for the general case is analogous: If $r_\varphi\geq 1$, one thinks in terms of graded vector spaces and block matrices instead.
\end{proof}

\begin{lemma}\label{lemma:localSystemDescent}
	Let $X$ be a real analytic manifold, $H\in\EbIL{X}$ and let $U\subseteq X$ be an open subset such that there is an isomorphism $\pi^{-1}L_U\otimes H \iso \e{L}{X}(\cL)$ for some local system $\cL$ on $U$ (extended by zero to $X$). Assume that $H$ has a $K$-lattice $H_K\in\EbRcIK{X}$. Then, there exists a local system $\cL_K$ on $U$ with $\cL\iso L_U\otimes_{K_U} \cL_K$ and $\pi^{-1}K_U\otimes H_K\iso \e{K}{X}(\cL_K)$.
\end{lemma}
\begin{proof}
	Since $\cL$ is a local system, there exists, for each $x\in U$, a subanalytic open neighbourhood $V$ such that 
	\begin{equation}\label{eq:isoLocSysBalls}
	\pi^{-1}L_V\otimes H\iso \e{L}{X}(L_V^r)= (\E^0_{V,L})^r
	\end{equation}
	for some $r\in\Z_{>0}$. By an argument similar to (but much simpler than) the one of Lemma~\ref{lemma:latticeE}, we see that the lattice of  $(\E^0_{V,L})^r$ given by the image of $\pi^{-1}K_V\otimes H_K$ under the isomorphism \eqref{eq:isoLocSysBalls} is -- up to an automorphism -- isomorphic to $(\E^0_{V,K})^r$. Hence, by composing with a suitable automorphism of $(\E^0_{V,L})^r$, the isomorphism \eqref{eq:isoLocSysBalls} comes from an isomorphism
	$$\pi^{-1}K_V\otimes H_K\iso (\E^0_{V,K})^r$$
	by applying $\piLK{X} (-)$. 
	This implies that $\pi^{-1}K_U\otimes H_K\iso \e{K}{X}(\cF_K)$ for some $\cF_K\in\Mod{K_U}$ by \cite[Proposition~3.8]{DKsh}.
	
	By the full faithfulness of $\e{L}{X}$ (see \cite[Proposition~4.7.15]{DK16}), we indeed have $L_X\otimes_{K_X}\cF_K\iso \cL$, and in particular $\cF_K$ is also a local system (cf.\ e.g.\ \cite[Lemma~2.13]{BHHS22} or \cite[Lemma~4.12]{Ho23GalShv}).
\end{proof}

\begin{thm}\label{thm:HLTdescent}
	Let $L/K$ be a field extension and let $H\in\EbRcIL{X}$. Assume that there exists $H_K\in\EbRcIK{X}$ which is a $K$-lattice for $H$, i.e.\ there is an isomorphism $H\iso \piLK{X} H_K$. If $H$ is of meromorphic normal form, then $H_K$ is of meromorphic normal form.
\end{thm}
\begin{proof}
	Let us check that $H_K$ is of meromorphic normal form:
	
	The condition (a) from Definition~\ref{def:HLTtype} is easily checked: Let us consider the natural morphism
	$\pi^{-1}K_{X^*}\otimes H_K \To H_K$.
	By assumption, applying the functor $\pi^{-1}L_X\linebreak\otimes_{\pi^{-1}K_X}(-)$ to it makes it an isomorphism and hence it is itself an isomorphism by Lemma~\ref{lemma:extConservative}.
	
	Property (b) is exactly what we proved in Lemma~\ref{lemma:localSystemDescent} (note that the datum of a local system on any $X^*_{\beps}$ for $0<\beps\ll 1$ determines one on $X^*$).
	
	Property (c) follows directly from Lemma~\ref{lemma:latticeE}: Given an isomorphism
	$$\xi\colon \pi^{-1}L_{S_\theta}\otimes H \iso \bigoplus_{\varphi\in \Phi_\theta}(\E^{\real \varphi}_{S_\theta,L})^{r_\varphi}$$
	as in Definition~\ref{def:HLTtype}(c) for $H$, we see that, due to the $L$-linearity of $\xi$, the object $\xi(\pi^{-1}K_{S_\theta}\otimes H_K)\subset \bigoplus_{\varphi\in \Phi_\theta}(\E^{\real \varphi}_{S_\theta,L})^{r_\varphi}$ is a $K$-lattice. We therefore see that, by Lemma~\ref{lemma:latticeE}, the isomorphism above comes, after composing it with an automorphism of the right-hand side, from an isomorphism
	$$\pi^{-1}K_{S_\theta}\otimes H_K \iso \bigoplus_{\varphi\in \Phi_\theta}(\E^{\real \varphi}_{S_\theta,K})^{r_\varphi}.$$
	This completes the proof.
\end{proof}

\subsection{Generalized monodromy data as gluing data for enhanced ind-sheaves}

Given an enhanced ind-sheaf $H\in\EbRcIk{X}$ of meromorphic normal form with respect to $D\subset X$, we can describe it purely in terms of linear algebra data. Let us recall the description of these so-called \emph{generalized monodromy data} here (in the case $\k=\C$, they are what is often called \emph{Stokes data}).

Since $H$ is of meromorphic normal form, any $p\in D$ admits data as in Remark~\ref{rem:finSectors}, in particular we can fix 
\begin{align}\label{eq:alphaFixed}
	\text{a finite collection of open sectors $S_1^p,\ldots,S_{k_p}^p$ and isomorphisms}\\
	\xi^p_j\colon \pi^{-1}\k_{S^p_j}\otimes H\overset{\sim}{\To} \bigoplus_{\varphi\in\Phi^p}(\E^{\real \varphi}_{S^p_j})^{r_\varphi} \quad \text{for $j\in\{1,\ldots,k_p\}$.}\notag
\end{align}

\begin{defi}
Let $H$ be of meromorphic normal form. We denote by $D_\mathrm{reg}\vcentcolon= \{p\in D\mid \Phi^p = \{0\}\}\subseteq D$ the set of \emph{regular singularities} of $H$. We also set $D_\mathrm{irr}\vcentcolon= D\setminus D_\mathrm{reg}$ and call its elements \emph{irregular singularities} of $H$.
\end{defi}

For any $p\in D_\mathrm{irr}$, we choose a local coordinate at $p$. We then choose a collection of real numbers $\beps=(\varepsilon_p)_{p\in D}$ with $0\leq \beps \ll 1$ such that the following holds: For any $p\in D_\mathrm{irr}$, we have $\varepsilon_p>0$ and $\overline{B}_{\varepsilon_p}(p)\setminus \{p\}\subset \bigcup_{j=1}^{k_p} S^p_j$, and for $p\in D_\mathrm{reg}$, we have $\varepsilon_p=0$ (so that $\overline{B}_{\varepsilon_p}(p)=\{p\}$, independently of the choice of a chart around $p$). 
This determines $X^*_{\beps}\subset X$. By Definition~\ref{def:HLTtype} (b) and (c), it is clear that there is a local system $\cL$ on $X^*$ such that we can choose an isomorphism
\begin{equation}\label{eq:gammaFixed}
\gamma\colon \pi^{-1}\k_{X^*_{\beps}} \otimes H \overset{\sim}{\To} \e{\k}{X}(\cL)
\end{equation}
(Note that this does not depend on the concrete choice of $\beps$. Note also that, here again, we tacitly consider the local system $\cL$ on $X^*$ extended by zero as a constructible sheaf on $X$.)

In view of wanting to use Lemma~\ref{lemma:homSumE}, we also fix the following data:
\begin{itemize}
	\item[(i)] For any $p$, a numbering (total order) on $\Phi^p$ and, for any $\varphi\in\Phi^p$, a numbering on the factors of the power $(\E^{\real\varphi}_{S^p_j})^{r_\varphi}$, so that there is a total order on the summands of $$\bigoplus_{\varphi\in\Phi^p}(\E^{\real\varphi}_{S^p_j})^{r_\varphi}=\bigoplus_{\varphi\in\Phi^p}\bigoplus_{m=1}^{r_\varphi} \E^{\real\varphi}_{S^p_j}$$ for any $j$.
	\item[(ii)] A point $x\in X^*_{\beps}$, a basis of the stalk $\cL_x$, and for each $p\in D_\mathrm{irr}$ a point $y_p\in S^p_1\cap X^*_{\beps}$ and a path in $X^*_{\beps}$ from $x$ to $y_p$. (Indeed, instead of choosing $y_p$ and such a path, it suffices to fix the homotopy class of this path in $X^*$ among all paths starting in $x$ and ending in an arbitrary point of the contractible set $S_1^p$.) Together with the above $\gamma$, this gives in particular an isomorphism $\pi^{-1}\k_{S^p_1\cap X^*_{\beps}} \otimes H \iso \bigoplus_{m=1}^r \E^0_{S^p_1\cap X^*_{\beps}}$ for any $p$.
\end{itemize}

We can then associate the following generalized monodromy data to $H$:
Let $p\in D_\mathrm{irr}$, then for any $j\in \{1,\ldots,k_p\}$ we get two (in general different) isomorphisms on the overlap $S^p_{j,j+1}$
$$\xi^p_j, \xi^p_{j+1}\colon \pi^{-1}\k_{S^p_{j,j+1}}\otimes H\overset{\sim}{\To} \bigoplus_{\varphi\in\Phi^p}(\E^{\real\varphi}_{S^p_{j,j+1}})^{r_\varphi}.$$
(These are induced by $\zeta^p_j$ and $\zeta^p_{j+1}$ above, and we denote them by the same symbols.)
From these, we get an automorphism (a transition or gluing automorphism)
$$\sigma^p_j\vcentcolon=\xi^p_{j+1} \circ (\xi^p_j)^{-1}\in \mathrm{End}\Big( \bigoplus_{\varphi\in\Phi^p}(\E^{\real\varphi}_{S^p_{j,j+1}})^{r_\varphi} \Big),$$
which is represented by an invertible square matrix $\Sigma^p_j\in \k^{r\times r}$ by Lemma~\ref{lemma:homSumE}.

Similarly, on the overlap $U^p_{\beps}\vcentcolon=S^p_1\cap X^*_{\beps}$, we get two isomorphisms
$$\xi^p_1,\gamma\colon \pi^{-1}\k_{U^p_{\beps}} \otimes H \overset{\sim}{\To} (\E^0_{U^p_{\beps}})^r$$
(note that $\E^{\real\varphi}_{U^p_{\beps}}\iso \E^0_{U^p_{\beps}}$). Then the composition
$$c_p\vcentcolon= \gamma\circ(\xi^p_1)^{-1}\in\mathrm{Hom}\big( (\E^0_{U^p_{\beps}})^r , (\E^0_{U^p_{\beps}})^r\big)$$
is again represented by an invertible square matrix $C_p\in\k^{r\times r}$.

\begin{defi}\label{def:GenMonData}
	The collection of data consisting of
	\begin{itemize}
		\item the subset $D_\mathrm{irr}\subseteq D$,
		\item the open sectors $S_j^p\subset X$ for any $p\in D_\mathrm{irr}$ and any $j\in\{1,\ldots,k_p\}$,
		\item the set $\Phi^p\subset \cO_X(S^p_1)$ together with a total order for every $p\in D_\mathrm{irr}$, and the number $r_\varphi\in\Z_{>0}$ for any $\varphi\in\Phi^p$,
		
		\item the matrices $\Sigma_j^p\in\k^{r\times r}$ for any $p\in D_\mathrm{irr}$ and any $j\in\{1,\ldots,k_p\}$,
		\item the matrix $C_p\in\k^{r\times r}$ for any $p\in D_\mathrm{irr}$,
		\item the local system $\cL\in\Mod{\k_{X^*}}$ on $X^*$, as well as the point $x\in X$, the basis of $\cL_x$ and the homotopy types of the paths from \textnormal{(ii)}
	\end{itemize}
is called \emph{generalized monodromy data} associated to $H$ with respect to all the choices \eqref{eq:alphaFixed}, \eqref{eq:gammaFixed}, \textnormal{(i)} and \textnormal{(ii)} fixed above.
\end{defi}
Of course, the generalized monodromy data just defined highly depend on the choices made (cf.\ also Remark~\ref{rem:genMonCan}).

The following statement is easily proved.
\begin{lemma}
	The generalized monodromy data determine $H$ up to isomorphism. More precisely, assume we are given $D\subset X$, an object $H\in\EbRcIk{X}$ of meromorphic normal form at $D$ and generalized monodromy data for $H$ with respect to certain choices as above. Then, if $H'\in\EbRcIk{X}$ is of meromorphic normal form at $D$ and has generalized monodromy data (for certain choices as above) that coincide with that of $H$ (meaning that the local systems associated to $H$ and $H'$ are isomorphic via an isomorphism respecting the given bases at $x$, and all the other data are the same for $H$ and $H'$), then $H\iso H'$.
\end{lemma}

\begin{rem}\label{rem:genMonCan}
	The matrices $\Sigma^p_j$ are what is usually called ``Stokes matrices'' in the context of solutions of differential equations. They describe the behaviour of the solutions around the singularity. On the other hand, the matrices $C_p$ are similar to what is usually referred to as ``connection matrices''. They give the relation between solutions around a singularity with the generic solutions away from the singularities.
	
	Let us also note that there are ways to make the choice and size of the sectors $S^p_j$ as well as the isomorphisms $\xi^p_j$ more canonical, which makes the definition of generalized monodromy data less ambiguous and more natural, but we will not insist on these choices here and only use the existence of these sectors and isomorphisms (see e.g.\ \cite{BoalchSMID,BoalchTop} and references therein for a thorough study of intrinsic generalized monodromy data for meromorphic connections). Let us just remark that with appropriate choices of sectors and isomorphisms, for example in the case of unramified exponents of one level, one gets Stokes matrices with a certain block-triangular structure and the diagonal blocks of all but one Stokes matrix around each point $p\in D$ are identities.
\end{rem}

\begin{defi}
	Let $L$ be a field and let $K\subset L$ be a subfield. Let $H\in\EbRcIL{X}$ be of meromorphic normal form. We say that generalized monodromy data associated to $H$ (as in Definition~\ref{def:GenMonData}) \emph{have entries in $K$} if all the matrices $\Sigma^p_j\in L^{r\times r}$ and $C_p\in L^{r\times r}$ have entries in $K$ and the local system $\cL\in\Mod{L_{X^*}}$ is defined over $K$, i.e.\ there exists a local system $\cL_K\in\Mod{K_{X^*}}$ and an isomorphism $\cL\iso L_X\otimes_{K_X}\cL_K$. 
	
	We say that $H$ \emph{admits generalized monodromy data with entries in $K$} (or \emph{there exist generalized monodromy data for $H$ with entries in $K$}) if there exist suitable choices with respect to which the generalized monodromy data associated to $H$ have entries in $K$.
\end{defi}

From what we proved above, we can now deduce the following statement.
\begin{cor}\label{prop:HLTmonodromyDescent}
	Let $L/K$ be a field extension, and let $H\in \EbRcIL{X}$ be of meromorphic normal form. Assume that $H$ has a $K$-structure $H_K\in\EbRcIK{X}$. Then there exist generalized monodromy data for $H$ with entries in $K$.
\end{cor}
\begin{proof}
	By Theorem~\ref{thm:HLTdescent}, if $H=\piLK{X} H_K$, then $H_K$ is also of meromorphic normal form. Therefore, $H$ admits generalized monodromy data that are induced by generalized monodromy data of $H_K$, and hence the matrices $\Sigma^p_j$, $C_p$ have entries in $K$ and also the local system is defined over $K$.
\end{proof}

\subsection{Meromorphic connections and Stokes data}
Let now $\cM$ be a meromorphic connection on $X$ with poles at $D$, i.e.\ a holonomic $\cD_X$-module such that $\cM(*D)\iso \cM$ and $\singsupp \cM=D$.

We briefly recall why $\SolE{X}(\cM)$ is of meromorphic normal form:

First of all, $\cM(*D)\iso \cM$ implies that $\pi^{-1}\C_{X^*}\otimes\SolE{X}(\cM)\iso \SolE{X}(\cM)$.

Since a meromorphic connection is generically an integrable connection, we also have, for any $0<\beps\ll 1$, an isomorphism
$$\pi^{-1}X^*_{\beps}\otimes \SolE{X}(\cM)\iso \e{\C}{X}(\cL),$$
where $\cL$ is the local system of solutions $\cL\vcentcolon= \Sol{X^*_{\beps}}(\cM|_{X^*_{\beps}})$, extended by $0$ to $X$.

Let $p\in D$, then the Levelt--Turrittin theorem gives us a decomposition (up to ramification) of the formalization of the stalk of $\cM$ at $p$
$$(\rho^*\cM)\hat{|}_p \iso \big(\bigoplus_ {i\in I} \cE^{\varphi_i} \tensOD \cR_i\big)\hat{|}_p,$$
where $\rho(t)=t^n$ is a ramification map in a small neighbourhood of $p$, the $\varphi_i(z^{-1})\in z^{-1}\C[z^{-1}]$ are (pairwise distinct) polar parts of Laurent series in a local coordinate $z$ at $p$ and the $\cR_i$ are regular holonomic $\cD_X$-modules.

By the Hukuhara--Turrittin theorem, locally on sufficiently small sectors around $p$, this decomposition lifts to an analytic decomposition of $\cM$, which is usually formulated as a statement on the real blow-up space of $X$ at the points of $D$. We will not go into too much detail here, and rather refer to the existing literature on asymptotic expansions in this context (see e.g.\ \cite{Was65}, \cite{Mal91}).

What we need is the following consequence on the level of the enhanced ind-sheaf associated to $\cM$: Let $p\in D$ and let $z$ be a local coordinate of $X$ at $p$. Then for any direction $\theta\in\R/2\pi\Z$, there exists a sufficiently small sector $S=S_\theta(r,\delta)$ and a finite set $\Phi$ of holomorphic functions on $S$ such that there is an isomorphism
$$\pi^{-1}\C_{S}\otimes \SolE{X}(\cM)\iso \bigoplus_{\varphi\in\Phi}(\E^{\real\varphi}_{S})^{r_\varphi}.$$

\begin{rem}
	The right-hand side can be described more explicitly: Concretely, $\Phi$ is the set of functions $\varphi_i(\zeta_n^j z^{1/n})$ for all $i\in I$ and $j\in \{1,\ldots,n-1\}$, where $\zeta_n\vcentcolon= e^{2\pi i/n}$ is a primitive $n$-th root of unity and $z^{1/n}$ is the choice of an $n$-th root on $S$ of a local coordinate $z$ around $p$. In particular, the functions $\varphi$ have Puiseux series expansions with a pole at $p$. The $r_\varphi$ are the ranks of the corresponding $\cR_i$.
\end{rem}

Indeed, there is the following equivalence (see \cite[Proposition~4.6]{MocCurveTest2}, \cite[Proposition~6.2.4]{DKmicrolocal}).
\begin{prop}\label{prop:MeroConnNormalForm}
	The category of meromorphic connections on $X$ with poles at $D$ is equivalent, via the functor $\SolE{X}$, to the subcategory of $\EbRcIC{X}$ consisting of objects of meromorphic normal form with respect to $D$.
\end{prop}
\begin{proof}
	The functor is certainly fully faithful by \cite[Theorem~9.5.3]{DK16}. It therefore remains to show that it is essentially surjective. Let therefore $H$ be of meromorphic normal form.
	The statement is proved on discs in \cite[Lemma~4.8]{MocCurveTest2}, so on suitable small discs $B_p$ around any $p\in D$, we can find $\cM_p$ such that $\SolE{B_p}(\cM_p)\iso H|_{B_p}$. Furthermore, on $X^*_{\beps}$ we can certainly find (by the regular Riemann--Hilbert correspondence) a $\cD_{X^*_{\beps}}$-module $\cM_{\beps}$, locally free over $\cO_{X^*_{\beps}}$, such that $\SolE{X^*_{\beps}}(\cM_{\beps})\iso H|_{X^*_{\beps}}$. All these D-modules glue to a single meromorphic connection on $X$ with poles at $D$ (cf., e.g., the argument in \cite[Proposition~2.17]{HoDiss}).
\end{proof}

We can now draw two direct consequences of our studies above. The first one follows directly from Corollary~\ref{prop:HLTmonodromyDescent}.
\begin{cor}\label{cor:StokesMatDescent}
	Let $\cM$ be a meromorphic connection with poles at $D$. If $\SolE{X}(\cM)$ has a $K$-structure $H_K\in\EbRcIK{X}$, then it admits generalized monodromy data defined over $K$, and in particular $\cM$ admits Stokes and connection matrices with entries in $K$.
\end{cor}

\begin{rem}
	This gives in particular an alternative proof of \cite[Theorem~5.4]{BHHS22}, and in a certain sense it is more natural since it does not leave the setting of enhanced ind-sheaves. (Another difference is that we do not pull back via the ramification before studying the Stokes matrices here.) In loc.~cit., we studied the special case of a hypergeometric system $\cM$, and the $K$-structure on $\SolE{X}(\cM)$ was used to derive properties of the Stokes filtration associated to $\cM$, from which we could then conclude the desired statement about the Stokes matrices.
	
	Going through the notion of Stokes filtration can indeed also be an approach in general: In \cite{DK23}, the authors give a functorial way of associating to an enhanced perverse ind-sheaf a Stokes-filtered local system (cf.\ \cite[Definition~4.1]{DK23} for a local version of such a functor). One could therefore derive a $K$-structure of the Stokes-filtered local system (in the sense of \cite{MocBetti}, for example) from a $K$-structure of the enhanced ind-sheaf and then conclude that the monodromy data extracted from the Stokes filtration can be defined over $K$. Knowing this relation between enhanced ind-sheaves and Stokes filtrations, one can argue that the statement of Corollary~\ref{cor:StokesMatDescent} is not entirely new. 
	Compared to the above ideas, our approach to the proof is, however, more direct and intrinsic to the theory of $\R$-constructible enhanced ind-sheaves, which is the main interest of this section.
	
	Lastly, let us remark that, more generally, $K$-structures for holonomic D-modules in any dimension have been studied in \cite{MocBetti}, and it seems reasonable to expect that a holonomic D-module has a $K$-structure in the sense of loc.~cit.\ if and only if $\SolE{X}(\cM)$ has a $K$-structure in our language.
\end{rem}

As a kind of upshot of all our considerations in this article, we get the following consequence of Theorem~\ref{prop:DREconj}, Theorem~\ref{thm:GaloisDescentEb} and Corollary~\ref{cor:StokesMatDescent}. Note, however, that our results from Sections~\ref{sec:GaloisDescentEb} and \ref{sec:Monodromy} are valid in greater generality and not restricted to the case of the field extension $\C/\R$. The results of the present section do not even require finite Galois extensions, while the results of Sections~\ref{sec:KashiwaraConj} and \ref{sec:GaloisDescentEb} are valid in any dimension.
\begin{cor}
	Let $X$ be a Riemann surface and let $\cM$ be a meromorphic connection with an isomorphism $\varphi\colon \cM\to c(\cM)$ such that $c(\varphi)\circ \varphi = \id_\cM$, then $\SolE{X}(\cM)$ has an $\R$-structure and $\cM$ admits generalized monodromy data with entries in $\R$.
\end{cor}

\end{document}